\documentclass[a4paper, reqno, 10pt]{amsart}
\usepackage{amssymb}
\usepackage{extarrows}
\usepackage{bm}
\usepackage[centering]{geometry}
\usepackage{enumitem}
\usepackage{longtable}
\usepackage{multirow}
\usepackage{array}
\usepackage{etex}
\usepackage{pictex}
\usepackage{graphics}

\usepackage[all]{xy}

\numberwithin{equation}{section}
\begin{document}

\newtheorem{lemdef}{Lemma-Definition}[section]
\newtheorem{theorem}{Theorem}[section]
\newtheorem{corollary}[theorem]{Corollary}
\newtheorem{lemma}[theorem]{Lemma}
\newtheorem{proposition}[theorem]{Proposition}
\newtheorem{definition}[theorem]{Definition}
\newtheorem{remark}[theorem]{Remark}
\newtheorem{examples}[theorem]{Examples}
\newtheorem{fact}[theorem]{Fact}
\def\Id{{\rm Id}}
\def\rad{{\rm rad}}

\title[Bi-Frobenius algebra structure on quantum complete intersections]{Bi-Frobenius algebra structure \\ on quantum complete intersections}
\author[Hai Jin, Pu Zhang]{Hai Jin, \ \ Pu Zhang$^*$  \\ \\  School of Mathematical Sciences \\
Shanghai Jiao Tong University,  \ Shanghai 200240, \ China }
\thanks{$^*$ Corresponding author}
\thanks{pzhang$\symbol{64}$sjtu.edu.cn \ \ \ \  jinhaifyh$\symbol{64}$sjtu.edu.cn}
\begin{abstract}
  This paper is to look for bi-Frobenius algebra structures on quantum complete intersections over field $k$.
We find a class of comultiplications, such that if $\sqrt{-1}\in k$,
then a quantum complete intersection becomes  a bi-Frobenius algebra with comultiplication of this form if and only if all the parameters $q_{ij} = \pm 1$. Also, it is proved that if $\sqrt{-1}\in k$ then
a quantum exterior algebra in two variables admits a bi-Frobenius algebra structure if and only if the parameter $q = \pm 1$. While if $\sqrt{-1}\notin k$, then the exterior algebra with two variables admits no bi-Frobenius algebra structures.
We prove that the quantum complete intersections
admit a bialgebra structure if and only if it admits a Hopf algebra structure, if and only if it is commutative, the characteristic of $k$ is
a prime $p$, and every  $a_i$ a power of $p$. This also provides a large class of examples of bi-Frobenius algebras which are not bialgebras (and hence not Hopf algebras). In commutative case, other two comultiplications on complete intersection rings are given, such that they admit non-isomorphic bi-Frobenius algebra structures.
\end{abstract}

\maketitle

\section{\bf Introduction}

Frobenius algebras are of particular interest and have wide applications (see e.g. \cite{Y1996,K1999}),
their module categories are Frobenius categories, and the corresponding stable categories
are triangulated (\cite{Hap1988}).
A bi-Frobenius algebra $A$, introduced by Y. Doi and M. Takeuchi \cite{DT2000}, is both a Frobenius algebra and a Frobenius coalgebra,
together with a linear map $S: A\longrightarrow A$, called the antipode, which is an algebra anti-homomorphism
and also a coalgebra anti-homomorphism.

Finite-dimensional Hopf algebras are bi-Frobenius algebras, but the converse is not true.
A bi-Frobenius algebra is a Hopf algebra if and only if it is a bialgebra (\cite{Haim2007}). The theory of bi-Frobenius algebras is developed in \cite{Doi2002, Doi2004, Haim2007}.
Some results in Hopf algebras have been generalized to bi-Frobenius algebras, e.g.
Radford's $S^4$-formula (\cite{Radford1976}), see \cite{DT2000}.

One of problems in this theory is to find
bi-Frobenius algebras which are not Hopf algebras. In \cite{DT2000} a coalgebra structure  on \ $k[X]/\langle X^n\rangle$ is constructed such that it is a bi-Frobenius algebra but not a Hopf algebra.
Group-like algebras are introduced in \cite{Doi2004} such that they become cocommutative bi-Frobenius algebras.
This class of bi-Frobenius algebras are further studied in \cite{Haim2007}, \cite{Doi2010}, and \cite{WL2014}.
In particular, using group-like algebras, a class of bi-Frobenius algebras which are not Hopf algebras
are found in \cite{Haim2007}, where the antipode \ $S$ is still the inverse of \ ${\rm Id}$ in the convolution algebra; also
the complexified stable Green algebra is a group-like algebra (\cite{WL2014}).
Using quiver approach,
some bi-Frobenius algebras which are not Hopf algebras are constructed in \cite{WZ2004} and \cite{WC2007}.

Quantum complete intersections originate from Yu. I. Manin's quantum planes  (\cite{Manin1987}),
reach their present form by L. L. Avramov, V. N. Gasharov and I. V. Peeva (\cite{AGP1997}),
and are closely related to braided Hopf algebras via quantum linear spaces by N. Andruskiewitsch and H.-J. Schneider (\cite{AH1998,AH2002}).
They reveal many exotic homological and representation properties of algebras
(see e.g.,\cite{LS1994,Ringel1996,BGMS2005,BE2008,BO2008,BO20082,BE2011,Mar,YZ2021,RZ2020}).

Quantum complete intersections are Frobenius algebras. Thus, it is natural
to look for bi-Frobenius algebra structures on them. In this paper, we construct a class of comultiplications on an arbitrary quantum complete intersection \ $A({\bf q}, a_1, \cdots, a_n)$ over field  $k$; and prove that
if $\sqrt{-1}\in k$, then it becomes  a bi-Frobenius algebra, with any comultiplication of this form, if and only if all the parameters $q_{ij} = \pm 1$. See Theorem \ref{mainthm}.
Also, it is proved that if $\sqrt{-1}\in k$, then
the quantum exterior algebra in two variables admits a bi-Frobenius algebra structure (no restrictions on the form of comultiplication) if and only if \ $q = \pm 1$ \ (see Theorem \ref{bifroontwovar}).
While if $\sqrt{-1}\notin k$, then the exterior algebra with two variables admits no bi-Frobenius algebra structures \ (see Proposition \ref{twovar}).

We stress that the comultiplication we used in Theorem \ref{mainthm} is different from
the one in {\rm Example 1.9} of {\rm \cite{AH2002}}. See Remark \ref{difference}.

Using Kummer's theorem in number theory,  we prove that the quantum complete intersections over filed $k$
admit a bialgebra structure if and only if it admits a Hopf algebra structure, if and only if it is commutative, and the characteristic of $k$ is
a prime $p$, and every  $a_i$ a power of $p$. See Theorem \ref{qcibialgebra}. This also gives a large class of examples of bi-Frobenius algebras which are not bialgebras (and hence not Hopf algebras). See Corollary \ref{bifrononbi}.

In commutative case, we find other two comultiplications on complete intersection rings, such that they admit non-isomorphic bi-Frobenius algebra structures. See Theorems \ref{pathcoalg} and \ref{2truncatedpolyalg}.

\section{\bf Preliminary}

\subsection{\bf Frobenius algebras and Frobenius coalgebras.}

Let $A$ be a finite-dimensional algebra over field \ $k$. Then the $k$-dual $A^*=\mathrm{Hom}_k(A, \ k)$ has a canonical $A$-bimodule structure. Denote by
 $a\rightharpoonup f$ the left action of \ $a\in A$ on \ $f\in A^*$, i.e., $
(a\rightharpoonup f)(x)=f(xa), \ \forall \ x\in A$. Similarly, denote by \
 $f\leftharpoonup a$ the right action of \ $a$ on \ $f$, i.e., \ $
(f\leftharpoonup a)(x)=f(ax)$.

Recall that \ $A$ is {\it a Frobenius algebra} provided that there is an isomorphism $A\cong A^*$ of left $A$-modules, or equivalently, there is an isomorphism $A\cong A^*$ of right $A$-modules; and that
$A$ is {\it a symmetric algebra}, provided that there is an isomorphism $A\cong A^*$ of $A$-bimodules.

By an elementary fact in linear algebra, a Frobenius algebra can also be characterized as a pair  \ $(A, \ \phi)$, where \ $\phi\in A^*$,  such that \ $A^* = A \rightharpoonup \phi$ (or equivalently, \ $A^* = \phi \leftharpoonup A$).
Such a \ $\phi$ is called {\it a Frobenius homomorphism} of \ $A$.

Let \ $(C, \ \Delta, \ \varepsilon)$ be a finite-dimensional coalgebra over  $k$. All the tensor products are over $k$. We use the Heyneman-Sweedler notation \ $\Delta(c)=\sum c_1\otimes c_2$.
Then \ $C^*:=\mathrm{Hom}_k(C,\ k)$ is an algebra and \ $C$ has a canonical $C^*$-bimodule structure.
Denote by \ $f\rightharpoonup c$ the left action of \ $f\in C^*$ on \ $c\in C$,  i.e., \ $f\rightharpoonup c=\sum c_1f(c_2)$.  Similarly, \ $c\leftharpoonup f=\sum f(c_1)c_2$.

A finite-dimensional coalgebra \ $C$  is a  Frobenius coalgebra if there is an element \ $t\in C$ such that \ $C=t\leftharpoonup C^*$,
or equivalently, \ $C=C^*\rightharpoonup t$. Usually we denote  a  Frobenius coalgebra by a pair \ $(C, \ t)$.

It is clear that \ $(A, \ \phi)$ is a Frobenius algebra if and only if \ $(A^*, \ \phi)$ is a Frobenius coalgebra.

For a coalgebra structure with comultiplication \ $\Delta$ on an algebra \ $A$, an element \ $x\in A$ is {\it primitive},  if \ $\Delta(x)=1\otimes x+ x\otimes 1$, where \ $1$ is the identity.
All the primitive elements form a linear space,  denoted by \ $P(A)$.

\subsection{\bf Bi-Frobenius algebras.}
\begin{definition} \ {\rm (\cite{DT2000})} \ Let $A$ be a finite-dimensional algebra and coalgebra over a field $k$ with $t\in A$ and $\phi\in A^*$. Let $S: A\longrightarrow A$ be the $k$-linear map
defined by $S(a)= \sum\phi(t_{1}a)t_{2}.$  The quadruple   $(A, \phi, t, S)$, or simply, $A$ is  a bi-Frobenius algebra, if the following conditions are satisfied$:$

{\rm (i)} \ The counit $\varepsilon: A\longrightarrow k$ is an algebra homomorphism, and the identity $1_A$ is a group-like element$;$

{\rm (ii)} \ $(A, \ \phi)$ is a Frobenius algebra, and \ $(A, \ t)$ is a Frobenius coalgebra$;$

{\rm (iii)} \  $S$ is an algebra anti-homomorphism, and $S$ is a coalgebra anti-homomorphism  $($i.e., $\varepsilon \circ S =\varepsilon$, and \  $\Delta(S(a)) = \sum S(a_{2})\otimes S(a_{1}), \ \forall \ a\in A)$.

In this case, $S$ is called the antipode of bi-Frobenius algebra  $(A, \phi, t, S)$.
\end{definition}

Let \ $(A, \ \phi, \ t,  \ S)$ and \ $(A', \ \phi', \ t', \ S')$ be bi-Frobenius algebras. A linear map \ $F: A\longrightarrow A'$ is  {\it a homomorphism of bi-Frobenius algebra} \ (\cite{Doi2002}), if \ $F$ is an algebra homomorphism and a coalgebra homomorphism, such that \ $F\circ S  = S' \circ F$.

An algebra $A$ is {\it augmented}, if there is an algebra homomorphism \ $\varepsilon: A\longrightarrow k$.
For an augmented algebra \ $(A, \varepsilon)$, an element \ $t\in A$ is  {\it a right} (respectively, {\it left}) {\it integral} of \ $A$ if \ $t x=t \varepsilon(x)$ (respectively, \ $x t= \varepsilon(x)t$), for all \ $x\in A$.
If the left integral space coincides with the right integral space, then \  $A$ is called {\it unimodular}.

Dually, a coalgebra \ $C$ is {\it coaugmented}, if it has a group-like element \ $h\in C$. For a coaugmented coalgebra \ $(C, \ h)$,
an element \ $\phi\in C^*$ is {\it a right}  (respectively, {\it left})  {\it cointegral} of
\ $C$ if \ $x \leftharpoonup\phi =\phi(x)h$  (respectively, \ $\phi \rightharpoonup x= h\phi(x)$), for all \ $x\in C$.
If the left cointegral space  coincides with the right cointegral space, then $C$ is called {\it counimodular}.

\begin{lemma} \label{dt} $(${\rm \cite{DT2000, Haim2007}}$)$  \ $(1)$ \ If \ $A$ is an augmented Frobenius algebra, then the space of right integrals is one-dimensional.

$(2)$ \ If \ $C$ is a Frobenius coalgebra with a group-like element, then the space of right cointegrals is one-dimensional.

$(3)$ \  If \ $(A, \ \phi, \ t, \ S)$ is a bi-Frobenius algebra, then \ $t$ is a right integral and \ $\phi$ is a right cointegral, and thus
the space of right integral is $kt$, the space of right cointegral is $k\phi$.

$(4)$ \  If \ $(A,\ \phi,\ t,  S)$ is a bi-Frobenius algebra, then $(A,  \ c\phi, \ \frac{1}{c}t,  S)$ is also a bi-Frobenius algebra  for any nonzero element \ $c\in k$.
\end{lemma}

We need the following observation in \cite{Doi2002}.

\begin{lemma}\label{lemDoi2002} {\rm (\cite{Doi2002})} \  Let A be a finite-dimensional algebra and coalgebra satisfying the condition \ {\rm (i)} in {\rm Definition 2.1}.
If there are \ $\phi\in A^*$ and $t\in A$ such that the $k$-linear map
  \[
  S: A\longrightarrow A, \ \ a\mapsto \sum \phi(t_1a)t_2
  \]
  is an algebra anti-isomorphism and a coalgebra anti-isomorphism,  then \  $(A, \ \phi, \ t, \ S)$ is a bi-Frobenius algebra.
\end{lemma}

\subsection{\bf Quantum complete intersections.} \ Let $k$ be a field, \ $a_1, \cdots, a_n$ integers with all \ $a_i\ge  2$ and \ $n\ge 2$, and \ ${\bf q} = (q_{ij})$ an \ $n\times n$ matrix
over \ $k$ with \ $q_{ii}=-1$ and $q_{ij}q_{ji}=1$ for  \ $1\le i, j\le n$. {\it A quantum complete intersection} $A({\bf q}, a_1, \cdots, a_n)$ is the following \ $k$-algebra
\[k\langle x_1, \cdots, x_n\rangle/\langle x_i^{a_i}, \ x_ix_j+q_{ij}x_jx_i, \ 1\le  i, \ j\le  n\rangle.
\]
When all $a_i = 2$, it is called {\it a quantum exterior algebra}.
It is well-known that a quantum complete intersection is a local Frobenius algebra and it is symmetric if and only if $\prod\limits_{1\le i\le n} (-q_{i j})^{a_i-1} = 1$ for all $j$ (see P. A. Bergh \cite [Lemma 3.1]{Bergh2009}).

Denote by $\mathbb{N}_0$  the set of non-negative integers.  Put
$$V=\{ \mathbf{v} = (v_1, \cdots, v_n)\in \mathbb{N}_0^n \ | \ v_i\le a_i-1,  \ 1\le i\le n\}.$$
For $\mathbf{v}\in \mathbb{N}_0^n$, write $x_{\mathbf{v}} = x_1^{v_1}\cdots x_n^{v_n}$. Then $\mathcal B = \{x_{\mathbf{v}} \ | \ \mathbf{v}\in V\}$ is a basis of $A({\bf q}, a_1, \cdots, a_n)$. The dual basis of \ $A^*$  is \ $\{x_{\bf v}^* \ |  \ {\bf v}\in V\}$, where
\ $x_{\bf v}^*(x_{\bf u}) = \delta_{{\bf v}, {\bf u}}, \ \forall \ {\bf v}, {\bf u}\in V$.

The degree of $x_{\mathbf{v}}\in \mathcal B$ is defined as \ $|x_{\mathbf{v}}| = \sum\limits_{i=1}^n v_i$,
and the degree of $x_{\mathbf{u}} \otimes x_{\mathbf{v}}\in \mathcal B\otimes \mathcal B$ is defined as \ $|x_{\mathbf{u}}|+|x_{\mathbf{v}}|$.

Put \ ${\bf a} = (a_1, \cdots, a_n), \ \ {\bf 0} = (0, \cdots, 0), \ \ {\bf 1} = (1, \cdots, 1)$. For $1\le i\le n$, let \ ${\bf e_i}$ \ be the vector in \ $\mathbb{N}_0^n$ with the \ $i$-th component \ $1$ and other components \ $0$.
Endow a partial order on $\mathbb{N}_0^n$  by \ $\mathbf{u}\le  \mathbf{v}$ if and only if \ $u_i\le  v_i$ for all $1\le i\le n$. Then $\mathbf{v}\le \mathbf{a-1}$ for all $\mathbf{v}\in V.$

For $\mathbf{u, v} \in \mathbb{N}_0^n$, put \ ${\bf q}^{\langle \mathbf{u}|\mathbf{v}\rangle}= \prod\limits_{1\le i<j\le n}(-\frac{1}{q_{ij}})^{u_jv_i}$. Then
$$x_\mathbf{u}x_\mathbf{v}={\bf q}^{\langle \mathbf{u}|\mathbf{v}\rangle}x_{\mathbf{u+v}}.$$
See S. Oppermann  \cite[Section 2]{Oppermann2010}.

To look for bi-Frobenius algebra structures on $A = A({\bf q}, \ a_1, \cdots, a_n)$, it is helpful to know all the
Frobenius homomorphisms of  $A$. By the proof of Lemma 3.1 in  \cite {Bergh2009}, it is known that $x_\mathbf{a-1}^*$ is a Frobenius homomorphism. Thus,  if  \ $\phi\in A^*$ is a Frobenius homomorphism, then \ $\phi= x_\mathbf{a-1}^*\leftharpoonup  r$, where \ $r$ is an invertible element of $A$. See M. Lorenz \cite[1.2.5]{L2011}.  Since $A$ is a finite-dimensional local algebra,  $r = c'_{\mathbf{a-1}}1+\sum\limits_{\mathbf{v}\neq\mathbf{0}} c'_{\mathbf{a-1 - v}} x_\mathbf{v}\in A$ with each  \ $c'_{\mathbf{a-1 - v}} \in k$ is invertible if and only if $c'_{\mathbf{a-1}} \ne 0$. It follows that a Frobenius homomorphism $\phi$  is of the form
$$\phi = x_{\mathbf{a-1}}^*\leftharpoonup r = x_\mathbf{a-1}^* \leftharpoonup (c'_{\mathbf{a-1}}1+\sum\limits_{\mathbf{v}\neq\mathbf{0}} c'_{\mathbf{a-1 - v}} x_\mathbf{v}) = \sum\limits_{\mathbf{v}\in V} c_\mathbf{v}
\ x_\mathbf{v}^*$$ with $c_{\mathbf{a-1}} = c'_{\mathbf{a-1}}\neq 0$. This proves {\rm (1)} below.

\begin{lemma} \label{frobeniushomomorphism} \ {\rm (1)} \  All the Frobenius homomorphisms of \ $A = A({\bf q}, a_1, \cdots, a_n)$ are
$$\phi = \sum_{\mathbf{v}\in V} c_\mathbf{v}x_\mathbf{v}^* \ \ \mbox{where} \ c_\mathbf{v} \in k \  \mbox{and} \  c_{\mathbf{a-1}} \neq 0.$$

{\rm (2)} \ Put \ $t = x_\mathbf{a-1}$  and  \ $\phi  = (x_{{\bf a-1}})^*$.
Then \ $(A, \ \phi)$ is a unimodular, augmented Frobenius algebra, with the space of integrals \ $kt.$
\end{lemma}

\begin{proof} {\rm (2)} \  By  {\rm (1)}, \ $\phi$ is a Frobenius homomorphism.
Let \ $\varepsilon: A\longrightarrow k$ be the algebra homomorphism given by \ $\varepsilon(x_\mathbf{v})=\delta_{\mathbf{v,0}}$. Then \ $t \ x_{\bf v} =  x_{\bf a-1} \ x_{\bf v}= \delta_{\bf v, 0} \ t = t \ \varepsilon(x_{\bf v})$ for \ ${\bf v}\in V$.
Thus \ $t$ is a right integral of augmented Frobenius algebra \ $(A, \ \phi)$.
Similarly, $t$ is a left integral.
By Lemma \ref{dt}(1), the space of right (respectively, left) integrals is one-dimensional, thus $A$ is unimodular, and the space of integrals \ $kt.$
\end{proof}

\section{\bf Bi-Frobenius algebra structures on quantum complete intersections}

We will give a class of comultiplications on an arbitrary quantum complete intersection
\[A = A({\bf q}, a_1, \cdots, a_n) = k\langle x_1, \cdots, x_n\rangle/\langle x_i^{a_i}, \ x_ix_j+q_{ij}x_jx_i, \ 1\le  i, \ j\le  n\rangle
\]
and then prove that if \ $\sqrt{-1}\in k$, then \ $A$ becomes a bi-Frobenius algebra,
with comultiplication of the given form, if and only if \ $q_{ij} = \pm 1$ \ for \ $1\le i, j\le n.$

\subsection{\bf Coalgebra structures on quantum complete intersections.}
Put \ $g_{\mathbf{0, a-1}} = 1 = g_{\mathbf{a-1, 0}},$ and choose arbitrary \ $(\prod\limits_{1\le i\le n} a_i)-2$ \ nonzero elements \ $g_{\mathbf{v, a-1-v}}\in k$ for \ ${\bf v}\in V, \ {\bf v}\ne {\bf 0}, \
{\bf v}\ne {\bf a-1}$.
  Define \  $\Delta: A \longrightarrow A\otimes A$ and $\varepsilon: A\longrightarrow k$ to be $k$-linear maps as

$$\left\{
\begin{array}{l}
\varepsilon(x_\mathbf{v})=\delta_{\mathbf{v,0}};\\

\Delta(1)= 1\otimes 1; \\

\Delta(x_\mathbf{v})= 1\otimes x_\mathbf{v} +x_\mathbf{v}\otimes 1 \ \ \text{for all} \ \ {\bf 0}\ne \mathbf{v}\in V \ \text{with} \ \ \mathbf{v}\ne \mathbf{a-1}; \\

\Delta(x_\mathbf{a-1}) = \sum\limits_{{\bf v}\in V} \ g_{{\bf v}, {\bf a-1 -v}} \ x_\mathbf{v}\otimes x_\mathbf{a-1 -v}
\\ \hskip 40pt = 1\otimes x_\mathbf{a-1} +x_\mathbf{a-1}\otimes 1 + \sum\limits_{\substack {{\bf 0}\ne {\bf v}\in V\\ {\bf v}\ne {\bf a-1}}} \ g_{{\bf v}, {\bf a-1 -v}} \ x_\mathbf{v}\otimes x_\mathbf{a-1 -v}.
\end{array}
\right. \eqno{(3.1)}$$

\begin{lemma}\label{lemcoalg}
With the comultiplication $\Delta$ and counit $\varepsilon$ defined as $(3.1)$, $(A,\Delta,\varepsilon)$ is a coalgebra.
\end{lemma}

\begin{proof}
  One directly verifies that $(A, \ \Delta, \ \varepsilon)$ is a coalgebra, i.e., \ $(\Delta \otimes {\rm Id})\Delta = ({\rm Id} \otimes \Delta)\Delta$ and \ $(\varepsilon\otimes {\rm Id})\Delta\cong {\rm Id} \cong ({\rm Id}\otimes\varepsilon)\Delta$. For example,
\begin{align*}
(\Delta \otimes {\rm Id})\Delta (x_{\bf a-1})
= & 1\otimes 1\otimes x_{\bf a-1} + 1\otimes x_{\bf a-1} \otimes 1 + x_{\bf a-1}\otimes 1 \otimes 1 \\
&+\sum_{\substack{{\bf 0}\ne {\bf v}\in V \\ {\bf v}\ne {\bf a-1}}} g_{{\bf v}, {\bf a-1-v}}
(x_\mathbf{v}\otimes x_\mathbf{a-1-v}\otimes 1 + 1\otimes x_\mathbf{v}\otimes x_\mathbf{a-1-v} + x_\mathbf{v}\otimes 1\otimes x_\mathbf{a-1-v})\\
=& ({\rm Id}\otimes \Delta)\Delta (x_{\bf a-1}).
\end{align*}
\end{proof}

Put \ $t = x_\mathbf{a-1}$  and  \ $\phi  = (x_{{\bf a-1}})^*$. Let \ $S: A\longrightarrow A$ be the $k$-linear map given by
\[
S(a)=\sum \phi(t_1a)t_2  =\sum (x_{\bf a -1})^*(t_1a)t_2 \ \ \text{for all} \ \ a\in A.
\]
Using \ $x_{\bf a-1-v} \ x_{\bf v} = {\bf q}^{\langle \mathbf{a-1-v|v}\rangle} \ x_{\bf a-1}$ one has
$$S(x_\mathbf{v})=\sum (x_{\bf a -1})^*(t_1x_\mathbf{v})t_2 = \ g_{\mathbf{a-1-v,v}} \ {\bf q}^{\langle \mathbf{a-1-v|v}\rangle} \ x_\mathbf{v} \ \ \text{for} \ \mathbf{v}\in V.
$$
Explicitly,
$$\left\{
\begin{array}{l}
S(1)=1, \ \ \
S(x_\mathbf{a-1})=x_\mathbf{a-1};
\\ \\
S(x_\mathbf{v})= \ g_{\mathbf{a-1-v,v}} \ {\bf q}^{\langle \mathbf{a-1-v|v}\rangle} \ x_\mathbf{v}, \ \ \text{for} \ {\bf 0}\ne  \mathbf{v}\ne {\bf a-1}.
\end{array}
\right. \eqno (3.2)$$

The main result of this section is as follows.

\begin{theorem} \label{mainthm} \  If
\ $A = A({\bf q}, \ a_1, \ \cdots, \ a_n)$ admits a bi-Frobenius algebra structure,
where the comultiplication \ $\Delta$ and the counit $\varepsilon$
are of the form \ $(3.1)$, then \ $q_{ij}^2=1$ \ for \ $1\le i, j\le n.$

Conversely, assume that \ $\sqrt{-1}\in k$ and \ $q_{ij}^2=1$ \ for \ $1\le i, j\le n$. Then \
$(A, \ (x_{{\bf a-1}})^*, \ x_{{\bf a-1}}, \ S)$ is a bi-Frobenius algebra,
where the comultiplication \ $\Delta$ and the counit $\varepsilon$ are of the form \ $(3.1)$ for some
special nonzero elements \ $g_{\mathbf{v, a-1-v}}\in k$ with \ ${\bf v}\in V$ and \
$g_{\mathbf{0, a-1}} = 1 = g_{\mathbf{a-1, 0}},$ and \ $S$ is given by \ $(3.2)$.
\end{theorem}

\begin{remark} \ {\rm (1)} \ From the proof of {\rm Theorem \ref{mainthm}}, we will see that the condition \  $\sqrt{-1}\in k$ can be removed in some  special cases.
For example, in the commutative case, i.e., \ $q_{ij} = -1$ \ for \ $1\le i, j\le n$, the condition $\sqrt{-1}\in k$ is not necessary.

On the other hand, if \ $\sqrt{-1}\notin k$, then indeed there is a quantum complete intersection $A$, such that
\ $A$ admits no bi-Frobenius algebra structures, for any $\Delta$ and $\varepsilon$, even for $\Delta$ and $\varepsilon$ given in \ $(3.3)$. See {\rm Proposition \ref{twovar}} below.

{\rm (2)} \ All the  bi-Frobenius algebras involved in {\rm Theorem \ref{mainthm}} has the property of \ $S^4 = {\rm Id}$.
\end{remark}

\subsection{\bf Frobenius coalgebra structures on quantum complete intersections.}
To prove Theorem \ref{mainthm} we need some preparations.

\begin{lemma} \label{frobeniuscoalg} Put \ $t = x_\mathbf{a-1}$  and  \ $\phi  = (x_{{\bf a-1}})^*$.
Then \ $(A, \ t)$ is a counimodular Frobenius coalgebra with the space of cointegrals \ $k\phi$, where the comultiplication \ $\Delta$ and the counit \ $\varepsilon$ are given by $(3.1)$, and \ $g_{\mathbf{v, a-1-v}}\in k$ are arbitrary nonzero elements \ with \ ${\bf v}\in V$ and \
$g_{\mathbf{0, a-1}} = 1 = g_{\mathbf{a-1, 0}}.$

Moreover, \ dim $P(A) = (\prod\limits_{1\le i\le n}a_i) -2.$

\end{lemma}
\begin{proof} \ By Lemma \ref{lemcoalg}, $(A,\Delta,\varepsilon)$ is a coalgebra. For $f\in A^*$, since \ $t\leftharpoonup f = \sum f(t_1)t_2$, one has
\begin{align*}
&t\leftharpoonup 1^* =  x_{\bf a-1};\\
&t\leftharpoonup x_{\bf a-1}^* = 1; \\
&t\leftharpoonup x_{\bf v}^* = g_{\bf v, \ a-1-v} \ x_{\bf a-1-v} \ \ \text{with} \ {\bf 0}\ne \mathbf{v}\ne \mathbf{a-1}.
\end{align*}
Since \ $g_{\bf v, \ a-1-v} \ne 0$,  \ $A = t \leftharpoonup A^*$, i.e., \ $(A, \ t)$ is a Frobenius coalgebra.

By $1\leftharpoonup\phi = 0 = \phi(1)1,$  $x_{\bf v}\leftharpoonup\phi = 0= \phi(x_{\bf v})1$ for ${\bf 0}\ne {\bf v}\in V$ with ${\bf v}\ne {\bf a-1}$,
and $x_{\bf a-1}\leftharpoonup\phi = 1 = \phi(x_{\bf a-1})1$, we see that \ $\phi$ is a right cointegral. Similarly, $\phi$ is a left cointegral.
By Lemma \ref{dt}(2), the space of right (respectively, left) cointegrals is one-dimensional, thus $A$ is counimodular, and the space of cointegrals \ $k\phi.$

Moreover, by (3.1), all \ $x_{\bf v}$ with \ ${\bf v}\in V$ are primitive elements of \ $A$,
except \ ${\bf v} = {\bf 0}$ or \ ${\bf v} = {\bf a-1}$.
Then, it is easy to see \ dim $P(A) = (\prod\limits_{1\le i\le n}a_i) -2.$
\end{proof}

\begin{remark}\label{difference} The comultiplication \ $\Delta$ in {\rm (3.1)} is different from
the one given in {\rm Example 1.9} by Andruskiewitsch and Schneider {\rm \cite{AH2002}:}  by {\rm Lemma \ref{frobeniuscoalg}}, the dimension of the space of primitive elements is $(\prod\limits_{1\le i\le n}a_i) -2;$
while the dimension of the space of primitive elements of the algebra in
{\rm Example 1.9} of {\rm \cite{AH2002}} is just \ $n$.
\end{remark}

\subsection{\bf When $S$ is an algebra anti-homomorphism?}
With \ $\Delta$ and \ $\varepsilon$ given in $(3.1)$ and $S$ given in $(3.2)$, we look for
a sufficient and necessary condition such that \ $S$ is an algebra anti-homomorphism.

\begin{lemma} \label{antipode1} \ Let \ $\Delta$ and \ $\varepsilon$ be given by $(3.1)$, and \ $S: A\longrightarrow A$  the $k$-linear map given by $(3.2)$.
Then $S$ is an algebra anti-homomorphism of $A$ if and only if the following conditions are satisfied$:$
$$q_{i, j}^2=1, \ \ \forall \ 1\leq i, j\le  n \eqno (3.3)$$
$$g_{\mathbf{a-1-(u+v), u+v}}=g_{\mathbf{a-1-u, u}} \ g_{\mathbf{a-1-v, v}}, \ \ \forall \ {\bf u},  \ {\bf v}\in V, \ \ {\bf u} + {\bf v}\le {\bf a-1}.
\eqno (3.4)$$
\end{lemma}

\begin{proof} \ Since the actions of $S$ on elements in the basis $\mathcal B = \{x_{\bf v} \ | \ {\bf v}\in V\}$ are just scalar multiplications,
to show $S$ is an anti-homomorphism of algebra $A$, it is enough to show \ $S(x_{\bf u}x_{\bf v}) = S(x_{\bf v})S(x_{\bf u})$ for \ ${\bf u}, {\bf v}\in V, \ {\bf u} + {\bf v} \le {\bf a-1}$.
By (3.2) one has

\begin{align*}
S(x_\mathbf{u}x_\mathbf{v}) & = S({\bf q}^{\langle \mathbf{u|v}\rangle}x_{\bf u+v})= {\bf q}^{\langle \mathbf{u|v}\rangle}S(x_{\bf u+v})\\
&={\bf q}^{\langle \mathbf{u|v}\rangle}  {\bf q}^{\langle \mathbf{a-1-(u+v)|u+v}\rangle}g_{\mathbf{a-1-(u+v),u+v}} \ x_\mathbf{u+v}\\
S(x_\mathbf{u})&=  {\bf q}^{\langle \mathbf{a-1-u|u}\rangle}g_{\mathbf{a-1-u,u}} \ x_\mathbf{u}\\
S(x_\mathbf{v})&=  {\bf q}^{\langle \mathbf{a-1-v|v}\rangle}g_{\mathbf{a-1-v,v}} \ x_\mathbf{v}\\
S(x_\mathbf{v})S(x_\mathbf{u})&={\bf q}^{\langle \mathbf{a-1-u|u}\rangle}{\bf q}^{\langle \mathbf{a-1-v|v}\rangle} \ g_{\mathbf{a-1-u,u}} \ g_{\mathbf{a-1-v,v}} \ x_\mathbf{v}x_\mathbf{u}\\
  &={\bf q}^{\langle \mathbf{v|u}\rangle}{\bf q}^{\langle \mathbf{a-1-u|u}\rangle}{\bf q}^{\langle \mathbf{a-1-v|v}\rangle} \ g_{\mathbf{a-1-u,u}} \ g_{\mathbf{a-1-v,v}} \ x_\mathbf{u+v}.
\end{align*}

Since \begin{align*}
  {\bf q}^{\langle \mathbf{u|v}\rangle} {\bf q}^{\langle\mathbf{a-1-(u+v)|u+v}\rangle}
  &= \prod_{1\le i\le j\le n}(-\frac{1}{q_{ij}})^{u_jv_i+(a_j-1-u_j-v_j)(u_i+v_i)}\\
  &= \prod_{1\le i\le j\le n}(-\frac{1}{q_{ij}})^{(a_j-1-u_j)u_i+(a_j-1-v_j)v_i-u_iv_j}\\
  &= {\bf q}^{\langle \mathbf{a-1-u|u}\rangle} {\bf q}^{\langle \mathbf{a-1-v|v}\rangle} {\bf q}^{-\langle \mathbf{v|u}\rangle}
\end{align*}
where ${\bf q}^{-\langle \mathbf{v|u}\rangle}=\frac{1}{{\bf q}^{\langle \mathbf{v|u}\rangle}}$, it follows that $S(x_{\bf u}x_{\bf v}) = S(x_{\bf v})S(x_{\bf u})$ for \ $\forall \ {\bf u}, \ {\bf v}\in V, \ \ {\bf u} + {\bf v}\le {\bf a-1}$ if and only if
$$g_{\mathbf{a-1-(u+v), u+v}}=({\bf q}^{\langle\mathbf{v|u}\rangle})^2 \ g_{\mathbf{a-1-u, u}} \ g_{\mathbf{a-1-v, v}}, \ \ \forall \ {\bf u}, \ {\bf v}\in V, \ \ {\bf u} + {\bf v}\le {\bf a-1}.
\eqno(3.5)$$

Assume that (3.5) holds.  Exchanging the position of ${\bf u}$ and ${\bf v}$, one gets
\[
  ({\bf q}^{\langle\mathbf{u|v}\rangle})^2=({\bf q}^{\langle \mathbf{v|u}\rangle})^2, \ \ \ \ \mbox{i.e.,} \ \ \ \
  \prod_{1\le i\le j\le n}q_{ij}^{-2u_jv_i}=\prod_{1\le i\le j\le n}q_{ij}^{-2v_ju_i}.
\]
Taking $\mathbf{v=e_{i}, \ u=e_{j}}$ with $1\le  i<j\le  n$, one gets \ $q_{i, j}^2=1$ for all $1\leq i, j\le  n$. Then
$({\bf q}^{\langle\mathbf{u|v}\rangle})^2= \prod\limits_{1\le i\le j\le n}q_{ij}^{-2u_jv_i}=1$,
and by (3.6) one gets
$$g_{\mathbf{a-1-(u+v), u+v}}=g_{\mathbf{a-1-u, u}} \ g_{\mathbf{a-1-v, v}}, \ \ \forall \ {\bf u}, \ {\bf v}\in V, \ \ {\bf u} + {\bf v}\le {\bf a-1}.
$$
That is, the conditions (3.3) and (3.4) are satisfied.

Conversely,  assume that the conditions (3.3) and (3.4) are satisfied.
Since  \ $q_{i, j}^2=1$ for all $1\leq i, j\le  n$,
$({\bf q}^{\langle\mathbf{v|u}\rangle})^2= \prod\limits_{1\le i\le j\le n}q_{ij}^{-2u_iv_j}=1$, and then
(3.5) holds, i.e., $S$ is an anti-homomorphism of algebra.

In conclusion, $S$ is an anti-homomorphism of algebra if and only if the conditions (3.3) and (3.4) are satisfied. \end{proof}

\subsection{\bf When $S$ is a coalgebra  anti-homomorphism?}
Next, we give a necessary and sufficient condition such that
$S$ is a coalgebra anti-homomorphism.

\begin{lemma} \label{antipode2} \ \ Let \ $\Delta$ and \ $\varepsilon$ be given by $(3.1)$, and \ $S: A\longrightarrow A$  the $k$-linear map given by $(3.2)$. Then $S$ is a coalgebra anti-homomorphism  \ $(A, \ \Delta, \ \varepsilon)$ if and only if
$$g^2_{\mathbf{v,a-1-v}} = {\bf q}^{-\langle \mathbf{v|a-1-v}\rangle} \ {\bf q}^{-\langle \mathbf{a-1-v|v}\rangle}, \ \ \forall \ {\bf v}\in V.
\eqno{(3.6)}$$
\end{lemma}
\begin{proof} By (3.1) and (3.2) one has
\begin{align*}
\varepsilon(S(x_\mathbf{v}))& =\delta_{\mathbf{v,0}}=\varepsilon(x_\mathbf{v}), \ \ \forall \ \mathbf{v}\in V;
\\
\Delta(S(1)) & = 1\otimes 1 = S\otimes S\circ\tau\circ(\Delta(1)),
\end{align*}
where $\tau$ is the twist map.  For any ${\bf 0}\ne \mathbf{v\ne a-1}$, one has
\begin{align*}
  \Delta(S(x_\mathbf{v})) &= g_{\mathbf{a-1-v,v}} \ {\bf q}^{\langle \mathbf{a-1-v|v}\rangle} \ \Delta(x_\mathbf{v})\\
  &=g_{\mathbf{a-1-v,v}} \ {\bf q}^{\langle \mathbf{a-1-v|v}\rangle} (1\otimes x_\mathbf{v} +x_\mathbf{v}\otimes 1)\\
  &=S\otimes S\circ\tau\circ(\Delta(x_\mathbf{v})).
\end{align*}
For ${\bf v}=\mathbf{a-1}$, one has
\begin{equation*}
  \Delta(S(x_\mathbf{a-1}))    = \Delta(x_\mathbf{a-1})=1\otimes x_\mathbf{a-1} +x_\mathbf{a-1}\otimes 1 + \sum\limits_{\substack{{\bf 0\ne v}\in V\\{\bf v} \ne {\bf a-1}}} \ g_{{\bf v}, {\bf a-1 -v}} \ x_\mathbf{v}\otimes x_\mathbf{a-1 -v};
\end{equation*}
while
\begin{align*}
  S\otimes S\circ\tau\circ(\Delta(x_\mathbf{a-1})) &=1\otimes x_\mathbf{a-1} +x_\mathbf{a-1}\otimes 1 + \sum\limits_{\substack{{\bf 0}\ne {\bf v}\in V\\{\bf v} \ne {\bf a-1}}} \ g_{\mathbf{v, a-1-v}} \ S(x_\mathbf{a-1-v})\otimes S(x_\mathbf{v})\notag\\
  & = 1\otimes x_\mathbf{a-1} +x_\mathbf{a-1}\otimes 1 \notag\\
  & + \sum\limits_{\substack{{\bf 0}\ne {\bf v}\in V\\{\bf v} \ne {\bf a-1}}} \ g_{\mathbf{a-1-v, v}} \ g_{\mathbf{v, a-1-v}}^2 \ {\bf q}^{\langle \mathbf{a-1-v|v}\rangle} \ {\bf q}^{\langle \mathbf{v|a-1-v}\rangle} \ x_\mathbf{a-1-v}\otimes x_\mathbf{v}.
\end{align*}
Comparing the coefficient of $x_\mathbf{v}\otimes x_\mathbf{a-1-v}$ with \ ${\bf 0}\ne \mathbf{v}\neq {\bf a-1}$, one gets
$$g_{\mathbf{v, a-1-v}}= g_{\mathbf{v, a-1-v}} \ g^2_{\mathbf{a-1-v, v}} \ {\bf q}^{\langle\mathbf{v|a-1-v}\rangle} \ {\bf q}^{\langle \mathbf{a-1-v|v}\rangle}, \ \  \forall \ \mathbf{v}\in V, \ {\bf 0}\ne \mathbf{v}\neq {\bf a-1}.$$
That is $$g^2_{\mathbf{a-1-v, v}} = {\bf q}^{-\langle \mathbf{a-1-v|v}\rangle}{\bf q}^{-\langle \mathbf{v|a-1-v}\rangle}, \ \  \forall \ \mathbf{v}\in V, \ {\bf 0}\ne \mathbf{v}\neq {\bf a-1}$$
where \ ${\bf q}^{-\langle \mathbf{u|v}\rangle}=\frac{1}{{\bf q}^{\langle \mathbf{u|v}\rangle}}$, or equivalently,
$$g^2_{\mathbf{v, a-1-v}} = {\bf q}^{-\langle \mathbf{v|a-1-v}\rangle} \ {\bf q}^{-\langle \mathbf{a-1-v|v}\rangle}, \ \  \forall \ \mathbf{v}\in V, \ {\bf 0}\ne \mathbf{v}\neq {\bf a-1}.$$
Since this equality also holds for $\mathbf{v} = {\bf 0}$
and for $\mathbf{v} = {\bf a-1}$,   $S$ is a coalgebra anti-homomorphism of  \ $(A, \Delta, \varepsilon)$ if and only if (3.6) holds. \end{proof}

\subsection{\bf Existence of $\{g_{\bf v, a-1-v}\}$.}
Now, we prove the existence of \ $\prod\limits_{1\le i\le n}a_i$ \ nonzero elements \ $g_{\mathbf{v, a-1-v}}\in k$ with \ ${\bf v}\in V$ and \ $g_{\mathbf{0, a-1}} = 1 = g_{\mathbf{a-1, 0}}$, satisfying $(3.4)$ and $(3.6)$, under some conditions.

\begin{lemma} \label{antipode3} \ Assume that \ $\sqrt{-1}\in k$ and \ $q_{ij}^2 = 1$ for $1\le i, j\le n$. Then there  are indeed \ $\prod\limits_{1\le i\le n}a_i$ \ nonzero elements \ $g_{{\bf v}, {\bf a-1-v}}$ in $k$ with \ ${\bf v}\in V$ and \ $g_{\mathbf{0, a-1}} = 1 = g_{\mathbf{a-1, 0}}$, such that   \ $(3.4)$ and $(3.6)$ are satisfied, namely

$$g_{\mathbf{a-1-(u+v), u+v}}=g_{\mathbf{a-1-u, u}} \ g_{\mathbf{a-1-v, v}}, \ \ \forall \ {\bf u}, \ {\bf v}\in V, \ {\bf u} + {\bf v}\le {\bf a-1}, \eqno(3.4)$$

$$g^2_{\mathbf{v, a-1-v}} = {\bf q}^{-\langle \mathbf{a-1-v|v}\rangle}{\bf q}^{-\langle \mathbf{v|a-1-v}\rangle}, \ \ \forall \ {\bf v}\in V. \eqno(3.6)
$$

\noindent In particular, \ $g^2_{\mathbf{v, a-1-v}} = \pm 1,  \ \forall \  {\bf v}\in V.$
\end{lemma}
\begin{proof} \ Since \ $\sqrt{-1}\in k$ and \ $\prod\limits_{1\le j\le n}(-q_{ij})^{a_j-1} = \pm 1$, thus one can choose
\[
g_{\bf a-1-e_i,e_i} = h_i\prod_{1\le j\le n}(-q_{ij})^{\frac{a_j-1}{2}}\in k,\ \ \text{where} \ h_i\in \{1,-1\}, \ \   1\le i\le n.
\]

(If \ ${\bf q = -1}$, i.e., $q_{ij} = -1$ for \ $1\le i, j\le n$, then \ $\prod\limits_{1\le j\le n}(-q_{ij})^{a_j-1} = 1$ and \ $g_{\bf a-1-e_i,e_i} = h_i\prod\limits_{1\le j\le n}(-q_{ij})^{\frac{a_j-1}{2}}$ is always in \ $k$.
Thus, in this case the assumption \ $\sqrt{-1}\in k$ can be removed.)

Then, one takes \  $g_{\mathbf{a-1-v, v}}$ with \ ${\bf 0\ne v}\in V,\ {\bf v\ne a-1}$ as
$$g_{\mathbf{a-1-v,v}} =\prod_{1\le i\le n}(g_{\mathbf{a-1-e_i,e_i}})^{v_i}.
$$
Since \ $g_{\mathbf{a-1, 0}} =1$, it is clear that this also holds for ${\bf v = 0}$ (note that in this case all \ $v_i = 0$).

{\bf Claim:}  One can choose \ $h_i\in \{1, -1\}$ for \ $1\le i\le n$,  such that \
$g_{\mathbf{a-1-v,v}} =\prod\limits_{1\le i\le n}(g_{\mathbf{a-1-e_i,e_i}})^{v_i}$  is also true for ${\bf v = a-1}$, i.e.,
 $$\prod\limits_{1\le i\le n}(g_{\mathbf{a-1-e_i,e_i}})^{a_i-1} = 1.$$

In fact, notice that \ $q_{ij}^2=1$ and \ $q_{ij}q_{ji}=1$,  hence \  $q_{ij}=q_{ji}.$ Thus one has
$$g_{\bf a-1-e_i,e_i}^{a_i-1} = h_i^{a_i-1} \prod_{1\le j\le n}(-q_{ij})^{\frac{(a_i-1)(a_j-1)}{2}},\ \ 1\le i\le n.$$ Also note that $q_{ii} = -1$ for $1\le i\le n$. One has
\begin{align*} \prod_{1\le i\le n}g_{\bf a-1-e_i,e_i}^{a_i-1} & = \prod_{1\le i\le n} (h_i^{a_i-1}\prod_{1\le j\le n}(-q_{ij})^{\frac{(a_i-1)(a_j-1)}{2}})
  \\ & = (\prod_{1\le i\le n}h_i^{a_i-1})(\prod_{1\le i\le n} \prod_{1\le j\le n}(-q_{ij})^{\frac{(a_i-1)(a_j-1)}{2}})
  \\ & = (\prod_{1\le i\le n}h_i^{a_i-1} )(\prod_{1\le i,j\le n}(-q_{ij})^{\frac{(a_i-1)(a_j-1)}{2}})\\
    &=(\prod_{1\le i\le n}h_i^{a_i-1} )(\prod_{1\le i< j\le n}(-q_{ij})^{\frac{(a_i-1)(a_j-1)}{2}})(\prod_{1\le j< i\le n}(-q_{ij})^{\frac{(a_i-1)(a_j-1)}{2}})\\
    &=(\prod_{1\le i\le n}h_i^{a_i-1} )(\prod_{1\le i< j\le n}(-q_{ij})^{\frac{(a_i-1)(a_j-1)}{2}})(\prod_{1\le i< j\le n}(-q_{ji})^{\frac{(a_i-1)(a_j-1)}{2}})\\
    &= (\prod_{1\le i\le n}h_i^{a_i-1} )(\prod_{1\le i<j \le n}(-q_{ij})^{(a_i-1)(a_j-1)}).
\end{align*}
Since \ $q_{ij}=\pm 1$, it follows that \  $\prod\limits_{1\le i<j \le n}(-q_{ij})^{(a_i-1)(a_j-1)}=\pm 1$. Choose \ $h_i\in \{1, -1\}$ for \ $1\le i\le n$ as follows.

If \ $\prod\limits_{1\le i<j \le n}(-q_{ij})^{(a_i-1)(a_j-1)}=1$, then we choose \ $h_i=1$ for all $1\le i\le n$.

If \  $\prod\limits_{1\le i<j \le n}(-q_{ij})^{(a_i-1)(a_j-1)}=-1$,  then there is \ $i_0$ with \ $1\le i_0\le n$ such that $a_{i_0}$ is even, then we can choose \ $h_{i_0}=-1$ and $h_i=1$ for all \ $i\ne i_0$.

In the both cases, one has \ $\prod\limits_{1\le i\le n}(g_{\mathbf{a-1-e_i,e_i}})^{a_i-1} = 1.$ This proves {\bf Claim}.

 Up to now, we have already fixed \ $\prod\limits_{1\le i\le n} a_i$ \ nonzero elements \ $g_{\mathbf{v, a-1-v}}\in k$ in the way
$$g_{\mathbf{a-1-v,v}} =\prod_{1\le i\le n}(g_{\mathbf{a-1-e_i,e_i}})^{v_i}, \ \ \forall \ {\bf v}\in V. \eqno(3.7)$$
It remains to show that these nonzero elements \ $g_{\mathbf{v, a-1-v}}\in k$ with  \ ${\bf v}\in V$ satisfy $(3.4)$ and $(3.6)$.
Indeed, for any $\ {\bf u},  \ {\bf v}\in V, \ {\bf u} + {\bf v}\le {\bf a-1},$ by (3.7) one has
  \begin{align*}
    g_{\mathbf{a-1-(u+v), u+v}} &= \prod_{1\le i\le n} g_{\mathbf{a-1-e_i, e_i}}^{u_i+v_i}\\
    &= \prod_{1\le i\le n} g_{\mathbf{a-1-e_i, e_i}}^{u_i} \prod_{1\le i\le n} g_{\mathbf{a-1-e_i, e_i}}^{v_i}\\
    & = g_{\mathbf{a-1-u, u}} \ g_{\mathbf{a-1-v, v}}
  \end{align*}
i.e.,  (3.4) is satisfied. Also, for \ ${\bf v}\in V$ by (3.7) one has
\begin{align*}
  g_{\mathbf{v,a-1-v}}^2
  &=\prod_{1\le i\le n}(g_{\mathbf{a-1-e_i,e_i}}^2)^{a_i-1-v_i} \notag\\
  &=\prod_{1\le i\le n}(\prod_{1\le j\le n}(-q_{ij})^{a_j-1})^{a_i-1-v_i} \notag\\
  &=\prod_{1\le i,j\le n} (-q_{ij})^{(a_j-1)(a_i-1-v_i)}.
\end{align*}
On the other hand, \ $q_{ij}=q_{ji}$ for $1\le i,j\le n$. By $q_{ii} = -1$ for $1\le i\le n$, one has
\begin{align*}
  {\bf q}^{-\langle \mathbf{a-1-v|v}\rangle}{\bf q}^{-\langle \mathbf{v|a-1-v}\rangle}
  &=\prod_{1\le i<j\le n} (-q_{ij})^{(a_j-1-v_j)v_i}\prod_{1\le i<j\le n} (-q_{ij})^{(a_i-1-v_i)v_j}\notag\\
  &=\prod_{1\le i<j\le n} (-q_{ij})^{(a_j-1-v_j)v_i} \prod_{1\le j<i\le n} (-q_{ji})^{(a_j-1-v_j)v_i}\notag\\
   &=\prod_{1\le i<j\le n} (-q_{ij})^{(a_j-1-v_j)v_i} \prod_{1\le j<i\le n} (-q_{ij})^{(a_j-1-v_j)v_i}\notag\\
    &=\prod_{1\le i,j\le n} (-q_{ij})^{(a_j-1-v_j)v_i}.
\end{align*}
To compare  $$g_{\mathbf{v,a-1-v}}^2 =\prod_{1\le i,j\le n} (-q_{ij})^{(a_j-1)(a_i-1-v_i)}$$ with  $${\bf q}^{-\langle \mathbf{a-1-v|v}\rangle}{\bf q}^{-\langle \mathbf{v|a-1-v}\rangle} = \prod_{1\le i,j\le n} (-q_{ij})^{(a_j-1-v_j)v_i}$$ one has
\begin{align*}
  & \prod_{1\le i,j\le n} (-q_{ij})^{(a_j-1)(a_i-1-v_i)-(a_j-1-v_j)v_i} \\
=& \prod_{1\le i,j\le n} (-q_{ij})^{v_iv_j+(a_i-1)(a_j-1)-2v_i(a_j-1)}\\
=& \prod_{1\le i,j\le n} (-q_{ij})^{v_iv_j+(a_i-1)(a_j-1)} \\
= & \prod_{1\le i<j\le n} (-q_{ij})^{v_iv_j+(a_i-1)(a_j-1)} + \prod_{1\le j<i\le n} (-q_{ij})^{v_iv_j+(a_i-1)(a_j-1)}\\
= & \prod_{1\le i<j\le n} (-q_{ij})^{v_iv_j+(a_i-1)(a_j-1)} + \prod_{1\le i<j\le n} (-q_{ji})^{v_iv_j+(a_i-1)(a_j-1)}
\\  =& \prod_{1\le i<j\le n} (-q_{ij})^{2v_iv_j+2(a_i-1)(a_j-1)} = 1.
\end{align*}
This proves \ $g^2_{\mathbf{v, a-1-v}} = {\bf q}^{-\langle \mathbf{a-1-v|v}\rangle}{\bf q}^{-\langle \mathbf{v|a-1-v}\rangle}, \ \ \forall \ {\bf v}\in V,$
i.e., (3.6) is satisfied. This completes the proof.
\end{proof}

\subsection{\bf Proof of Theorem \ref{mainthm}. } Assume that $A$ admits a bi-Frobenius algebra structure  \ $(A, \ \phi, \ t, \ S)$, where the comultiplication \ $\Delta$ and the counit $\varepsilon$
are of the form \ $(3.1)$. By Lemma \ref{dt}(3), \ $t$ is a right integral of $A$. On the other hand, by Lemma \ref{frobeniushomomorphism}(2),
the space of right integrals of \ $A$ is \ $k \ x_{{\bf a-1}}$. It follows that $t = c \ x_{{\bf a-1}}$ for some $0\ne c\in k$.

By Lemma \ref{dt}(3), \ $\phi$ is a right cointegral of \ $A$. On the other hand, by Lemma \ref{frobeniuscoalg}, the space of right cointegrals of \ $A$ is \ $k \ (x_{{\bf a-1}})^*$. It follows that $\phi = d \ (x_{{\bf a-1}})^*$ for some $0\ne d\in k$.

By \ $$\Delta(t) = \Delta(cx_{{\bf a-1}}) = c(1_A\otimes x_{{\bf a-1}} + x_{{\bf a-1}}\otimes 1_A + \sum\limits_{\substack {{\bf 0}\ne {\bf v}\in V\\ {\bf v}\ne {\bf a-1}}} \ g_{{\bf v}, {\bf a-1 -v}} \ x_\mathbf{v}\otimes x_\mathbf{a-1 -v})$$ and $$1_A = S(1_A) = \sum \phi(t_1)t_2 = \sum d x_{{\bf a-1}}^*(t_1)t_2 = cd 1_A$$ one gets $d = \frac{1}{c}$.
Thus,  $(A, \ \frac{1}{c} (x_{{\bf a-1}})^*, \ cx_{{\bf a-1}}, \ S)$  is a bi-Frobenius algebra. By Lemma \ref{dt}(4),
$(A, \ (x_{{\bf a-1}})^*, \ x_{{\bf a-1}}, \ S)$  is also a bi-Frobenius algebra. In particular, \ $S$ is an algebra anti-homomorphism, and hence by Lemma \ref{antipode1}, \ $q_{ij}^2=1$ \ for \ $1\le i, j\le n$.

Conversely, assume that \ $\sqrt{-1}\in k$ and \ $q_{ij}^2=1$ \ for \ $1\le i, j\le n$. By Lemma \ref{frobeniushomomorphism}(2), $(A, \ (x_{{\bf a-1}})^*)$ is a Frobenius algebra.

Choose nonzero elements \ $g_{\mathbf{v, a-1-v}}\in k$ with \ ${\bf v}\in V$ and \
$g_{\mathbf{0, a-1}} = 1 = g_{\mathbf{a-1, 0}}$, such that \ $(3.4)$ and $(3.6)$ are satisfied.
Lemma \ref{antipode3} guarantees the existence of these $g_{\mathbf{v, a-1-v}}$'s.

Using these chosen $g_{\mathbf{v, a-1-v}}$'s, consider the $k$-linear maps \ $\Delta$ and \ $\varepsilon$, as defined in $(3.1)$. Then \ $\varepsilon$ is an algebra homomorphism. By Lemma \ref{frobeniuscoalg}, \ $(A, \ x_{{\bf a-1}})$ is a Frobenius coalgebra. Clearly, the identity \ $1$ is a group-like element.

By Lemmas \ref{antipode1} and \ref{antipode2},  the $k$-linear map $S: A\longrightarrow A$ given by $(3.2)$ is
an algebra anti-homomorphism of $A$ and a coalgebra anti-homomorphism of \ $(A, \ \Delta, \ \varepsilon)$. By definition, $(A, \ x_{{\bf a-1}}^*$, \ $x_{{\bf a-1}}, \ S)$  is a bi-Frobenius algebra.

This completes the proof. \hfill $\square$

By \ ${\bf q = -1}$ we mean \ $q_{ij} = -1$ for \ $1\le i, j\le n$. From the proof of Lemma \ref{antipode3} we see

\begin{corollary} \label{cormainthm} \  Complete intersection rings
\ $A = A({\bf -1}, \ a_1, \ \cdots, \ a_n)$ admit a bi-Frobenius algebra structure \
$(A, \ (x_{{\bf a-1}})^*, \ x_{{\bf a-1}}, \ S)$,
where the comultiplication \ $\Delta$ and the counit $\varepsilon$
are of the form \ $(3.1)$.
\end{corollary}

\section{\bf When a quantum complete intersection admits a Hopf algebra structure?} \  \ The following main result of this section shows that only in very special
cases quantum complete intersections can admit Hopf algebra structures.

\begin{theorem}\label{qcibialgebra} For the quantum complete intersection $A = A({\bf q},a_1, \cdots, a_n)$ over field $k$, the following are equivalent.

$(1)$ \ $A$ admits a Hopf algebra structure$;$

$(2)$ \ $A$ admits a bialgebra structure$;$

$(3)$ \ $A$ is commutative, and each $a_i$ is a positive power of $p$, where $p = {\rm char} k$ is a prime.
\end{theorem}

\subsection{\bf Kummer's theorem} \
To prove Theorem \ref{qcibialgebra}, we need a consequence of Kummer's theorem.

Let $p$ be a prime. Denote by $\nu_p(x)$ the $p$-adic valuation of a nonzero integer $x$, i.e.,
the largest non-negative integer $s$ such that $p^s$ divides $x$.
For $n\in \mathbb{N}_0$, write $n$ as the expansion in base $p$:
$$n = n_0+ n_1p + \cdots +n_rp^r$$ where $0\le n_0, \cdots, n_r \le p-1.$
For $n, m\in \mathbb{N}_0$ with $n\ge m$, write $n, \ m$ and $t = n-m$ as the expansion in base $p$, with the digits $n_j$, $m_j$ and $t_j$ as above, respectively.
Let $\epsilon_j = 1$ if there is a carry in the $j$-th digit when adding $m$ and $t$ in base $p$, and let $\epsilon_j=0$ otherwise.
Then Kummer's theorem claims that the $p$-adic valuation of binomial coefficient $\binom{n}{m}$ is
\[
\nu_p(\binom{n}{m})=\sum_{j\ge 0}\epsilon_j.\]
See e.g. \cite[Theorem 2.6.7]{Moll2012}.

\begin{lemma}\label{kummer} \ {\rm (E. Kummer)} \ Let $n$ be a positive integer and $p$ a prime.
Assume that $p\mid n$ with $\nu_p(n)=r$. Then \ $p\nmid\binom{n}{p^r}.$
\end{lemma}
\begin{proof} \ For convenience write $m = p^r$.  Then the expansion of $m$ in base $p$ is just
$m = p^r$, i.e., $$m_{j} = \delta_{j, r} = \begin{cases}0, & j\ne r; \\ 1, &j = r. \end{cases}$$
To know $\nu_p(\binom{n}{p^r})$, let $t = n-m = p^r(\frac{n}{m}-1)$. Then the expansion of $t$ in base $p$ looks like
$$t = t_rp^r + t_{r+1}p^{r+1} + \cdots + n_{r+s}p^{r+s}.$$
It is clear that $t_r\ne p-1$. Otherwise, $n = m+t$ is divided by $p^{r+1}$, contradicts $\nu_p(n)=r$.

Thus, when adding $m$ and $t$ in base $p$, there is no carry in $r$-th digit, and hence there are no carries in  any digit.
By Kummer's theorem,  $\nu_p(\binom{n}{p^r}) = 0,$ \ i.e., \ $p\nmid\binom{n}{p^r}.$ \end{proof}

\subsection{\bf A main lemma}

\begin{lemma} \label{bialgebra} \ If there is an $a_i$ such that it is not a power of {\rm char} $k$, then the quantum complete intersection $A({\bf q}, a_1, \cdots, a_n)$ admits no bialgebra structures.

In particular, if \ ${\rm char} \ k= 0$, then any quantum complete intersection
admits no bialgebra structures.
\end{lemma}
\begin{proof} \  Without loss of generality, one may assume that $a_1$ is not a power of {\rm char} $k$. Assume otherwise that \ $A= A({\bf q}, a_1, \cdots, a_n)$ has a
bialgebra structure, say, with comultiplication $\Delta$ and counit $\varepsilon$. Thus
 \ $\Delta$ is an algebra homomorphism.
Write
\[
\Delta(x_1)=\sum_{\mathbf{u,v}\in V} {a_{\mathbf{u,v}}}x_{\mathbf{u}}\otimes x_{\mathbf{v}} \ \ \mbox{with} \  \ a_{\mathbf{u,v}}\in k.
\]
Since \ $A$ is a finite-dimensional local algebra, the kernel of the algebra homomorphism \ $\varepsilon: A\longrightarrow k$ is the unique maximal ideal of $A$. Thus \ $\varepsilon(x_{\mathbf{u}}) = \delta_{\mathbf{u}, \mathbf{0}}$,  where $\delta$ is the Kronecker symbol. By $(\varepsilon\otimes {\rm Id})\Delta\cong{\rm Id}\cong({\rm Id}\otimes \varepsilon)\Delta$, one has
\[
    \sum_{\mathbf{v}\in V} a_{\mathbf{0,v}}x_{\mathbf{v}}= x_1= \sum_{\mathbf{u}\in V} a_{\mathbf{u,0}}x_{\mathbf{u}}.
\]
Thus $a_{\mathbf{0, v}}=\delta_{\mathbf{v,e_1}}$ and $a_{\mathbf{u,0}}=\delta_{\mathbf{u,e_1}}$, and hence
\[
    \Delta(x_1)=1\otimes x_1 + x_1\otimes 1 +  \sum_{\mathbf{u,v}\in V/\{\mathbf{0}\}} {a_{\mathbf{u,v}}}x_{\mathbf{u}}\otimes x_{\mathbf{v}}.
\]
Write
\begin{align*}\Delta(x_1)^{a_1} & =
(1\otimes x_1 + x_1\otimes 1 +  \sum\limits_{\mathbf {u, v}\in V/\{\mathbf{0}\}} {a_{\mathbf {u,v}}} x_{\mathbf {u}} \otimes x_{\mathbf {v}})^{a_1}
\\ & = (1\otimes x_1 + x_1\otimes 1)^{a_1} + \Sigma_1 \\ & = \sum_{1\le i \le a_1-1}\binom{a_1}{i}  x_1^i\otimes x_1^{a_1-i} + \Sigma_1 \end{align*}
and write $\Sigma_1$ as a $k$-combination of elements in $\mathcal B\otimes \mathcal B$.
Notice that the degree of each term in $x_1^i\otimes x_1^{a_1-i} \ (1\le i \le a_1-1)$ is $a_1$;
and the degree of each term in $\Sigma_1$ is bigger than or equal to
$$a_1-i + 2i = a_1 + i$$ with $1\le i\le a_1$. Thus the degree of each term in $\Sigma_1$ is bigger than $a_1$.
Since  \ $\Delta$ is an algebra homomorphism, \ $(\Delta x_1)^{a_1}=\Delta(x_1^{a_1})=0$. It follows that
$\binom{a_1}{i}=0, \ 1\le i\le a_1-1$. This is absurd if char $k=0$.

So, suppose that ${\rm char} \ k = p > 0$ and $a_1$ is not a power of $p$.  If $p \nmid a_1$, then  $\binom{a_1}{1} = a_1\ne 0$ in $k$, a contradiction.
If the largest positive integer $t$ such that $p^t\mid a_1$ is $s$, then by Lemma \ref{kummer},
$p\nmid \binom{a_1}{p^s}$, and hence $\binom{a_1}{p^s}\ne 0$ in $k$, again a contradiction. This completes the proof. \end{proof}

\subsection{\bf Proof of Theorem \ref{qcibialgebra}.} \ The implication $(1) \Longrightarrow (2)$ is clear.

$(2) \Longrightarrow (3)$: \ Suppose that $A$ admits a bialgebra structure for some comultiplication $\Delta$ and counit $\varepsilon$. By {\rm Lemma \ref{bialgebra}}, ${\rm char} k = p$ should be a prime, and each $a_i$ is a positive power of $p$.

As in the proof of {\rm Lemma \ref{bialgebra}}, for any $1\le i\le n$ one has
\[
   \Delta(x_i)=1\otimes x_i + x_i\otimes 1 + \sum_{\mathbf{u,v}\in V/\{\mathbf{0}\}} {a_{\mathbf{u,v}}}x_{\mathbf{u}}\otimes x_{\mathbf{v}}
\]
where $a_{\bf u,v}\in k$. Then for any $1\le i,j\le n$, one has
\begin{align*}
  \Delta(x_ix_j) &=\Delta(x_i)\Delta(x_j)=1\otimes x_ix_j+x_i\otimes x_j+x_j\otimes x_i+x_ix_j\otimes 1+\Sigma
\end{align*}
with some $\Sigma\in A\otimes A$. On the other hand, one has
\begin{align*}
    \Delta(x_ix_j) &= \Delta(-q_{ij}x_jx_i) = -q_{ij}\Delta(x_j)\Delta(x_i)\\
    &=-q_{ij}(1\otimes x_jx_i+x_i\otimes x_j+x_j\otimes x_i+x_jx_i\otimes 1+\Sigma')\\
    &=1\otimes x_ix_j-q_{ij}x_i\otimes x_j-q_{ij}x_j\otimes x_i+x_ix_j\otimes 1-q_{ij}\Sigma'
\end{align*}
with some $\Sigma'\in A\otimes A$. Thus
\begin{align*} & 1\otimes x_ix_j+x_i\otimes x_j+x_j\otimes x_i+x_ix_j\otimes 1+\Sigma \\ & = 1\otimes x_ix_j-q_{ij}x_i\otimes x_j-q_{ij}x_j\otimes x_i+x_ix_j\otimes 1-q_{ij}\Sigma'.\end{align*}
Since $\mathcal B\otimes \mathcal B$ is a set of basis of $A\otimes A$
and the degree of each term in $\Sigma$ and $\Sigma'$ is bigger than or equal to $3$ (see Subsection 2.3), one gets
\[
q_{ij}=-1, \ \ \ \forall\ 1\le i, j\le n
\]
i.e., $A$ is commutative.

$(3) \Longrightarrow (1)$: Assume that $A$ is commutative, and that $a_i = p^{r_i}$ for some positive integer $r_i$ for all $1\le i\le n$, where $p = {\rm char} k$. Note that
$\mathcal B = \{x_{\mathbf{v}} = x_1^{v_1}\cdots x_n^{v_n} \ | \ \mathbf{v}\in V\}$ is a basis of $A$, where $V=\{ \mathbf{v} = (v_1, \cdots, v_n)\in \mathbb{N}_0^n \ | \ v_i\le a_i-1,  \ 1\le i\le n\}.$ Define \  $\Delta: A \longrightarrow A\otimes A$, $\varepsilon: A\longrightarrow k$ and $S: A\longrightarrow A$ to be the $k$-linear maps as follows:
$$\left\{
\begin{array}{l}
\varepsilon(x_\mathbf{v})=\delta_{\mathbf{v,0}};\\
\Delta(x_\mathbf{v})=\prod\limits_{1\le i\le n}\sum\limits_{0\le t\le v_i}\binom{v_i}{t}x_i^t\otimes x_i^{v_i-t}; \\

S(x_{\bf v})=(-1)^{|{\bf v}|}x_{\bf v}, \\
\end{array}
\right.$$
It is clear that $\varepsilon$ is an algebra homomorphism. We claim that $\Delta$ is an algebra homomorphism.

In fact, by definition $\Delta(x_i)=1\otimes x_i +x_i\otimes 1$ for  $1\le i\le n$ and $\Delta(x_{\bf v})=\prod\limits_{1\le i\le n}(\Delta (x_i))^{v_i}$ for ${\bf v}\in V$. Thus for ${\bf u,v}\in V$, if ${\bf u+v}\in V$, then
$$\Delta (x_{\bf u+v})=\prod\limits_{1\le i\le n}\Delta(x_i)^{u_i+v_i}=\prod\limits_{1\le i\le n}\Delta(x_i)^{u_i}\Delta(x_i)^{v_i} = \Delta(x_{\bf u})\Delta(x_{\bf v}).$$
If ${\bf u+v}\notin V$, then $x_{\bf u}x_{\bf v}=0$, i.e., there is an integer $l$ with $1\le l \le n$ such that $u_{l}+v_{l}\ge a_{l}$.
Since {\rm char} $k=p$ is a prime and $a_{l}=p^{r_{l}}$,  it follows that
\[\Delta(x_l)^{a_l} = (1\otimes x_l +x_l\otimes 1)^{p^{r_l}}
= 1\otimes x_l^{p^{r_l}} +x_l^{p^{r_l}} \otimes 1 = 1\otimes x_l^{a_l} +x_l^{a_l} \otimes 1 = 0\]
where we have used the rule $(a+b)^{p^r} = a^{p^r} + b^{p^r}.$
Hence $\Delta(x_l)^{u_l+v_l} = 0$, since $u_{l}+v_{l}\ge a_{l}$.
Thus \begin{align*}
    \Delta(x_{\bf u}x_{\bf v})=0&=\Delta(x_l)^{u_l+v_l}\prod\limits_{1\le i\le n,\ i\ne l}\Delta(x_i)^{u_i+v_i}\\
    &= \prod\limits_{1\le i\le n}\Delta(x_i)^{u_i}\Delta(x_i)^{v_i}=\Delta(x_{\bf u})\Delta(x_{\bf v}).
\end{align*}
This proves the claim.

Since $\Delta$ is an algebra homomorphism, it follows that $\Delta\otimes {\rm Id}$ and ${\rm Id}\otimes\Delta$ are also algebra homomorphisms. One has
\begin{align*}
  (\Delta\otimes {\rm Id})\Delta(x_{\bf v}) &=\Delta\otimes {\rm Id}(\prod\limits_{1\le i\le n}\Delta (x_i)^{v_i})\\
  &=\prod\limits_{1\le i\le n}((\Delta\otimes {\rm Id})\Delta (x_i))^{v_i}\\
  &=\prod\limits_{1\le i\le n}(1\otimes 1\otimes x_i+ 1\otimes x_i\otimes 1+x_i\otimes 1\otimes 1 )^{v_i}\\
  &=\prod\limits_{1\le i\le n}(({\rm Id}\otimes \Delta)\Delta (x_i))^{v_i}\\
  &=({\rm Id}\otimes \Delta)\Delta(x_{\bf v}), \ \text{ for } {\bf v}\in V.
\end{align*}
This shows the coassociativity of $\Delta$. It is clear that the counitary property $(\varepsilon\otimes {\rm Id})\Delta\cong {\rm Id}\cong({\rm Id}\otimes \varepsilon)\Delta$ is satisfied. Thus $(A,\Delta,\varepsilon)$ is a coalgebra.

Finally, one can show that $S$ is an antipode. For any $\mathbf{v}\in V$,
\begin{align*}
    S*{\rm Id}(x_{\bf v})&=m\circ(S\otimes {\rm Id})\Delta(x_{\bf v}) \\
    &= m\circ(S\otimes {\rm Id})\prod_{1\le i\le n}\sum_{0\le t\le v_i}\binom{v_i}{t}x_i^t\otimes x_i^{v_i-t}\\
    &= \prod_{1\le i\le n}(\sum_{0\le t\le v_i}(-1)^t\binom{v_i}{t})x_i^{v_i}\\
    &= \varepsilon(x_{\bf v})1 ={\rm Id}*S(x_{\bf v}).
\end{align*}
Thus $A$ is a Hopf algebra. \hfill $\square$

\begin{corollary}\label{hopf} \ A quantum exterior algebra admits a bialgebra structure if and only if ${\rm char} \ k = 2$. In this case, it is a Hopf algebra.
\end{corollary}

\subsection{\bf A class of bi-Frobenius algebras which are not bialgebras.} \ Combined Theorem \ref{mainthm} with Theorem \ref{qcibialgebra}, one gets a large class of examples of bi-Frobenius algebras which are not bialgebras (and hence not Hopf algebras).

\begin{corollary} \label{bifrononbi} \ Assume that \ $\sqrt{-1}\in k$ and \ $q_{ij}^2=1$ \ for \ $1\le i, j\le n$. Then the quantum complete intersection $A = A({\bf q}, \ a_1, \ \cdots, \ a_n)$ admits a bi-Frobenius algebra structure$;$ moreover, if there is an $a_i$ such that it is not a power of \ ${\rm char} \ k$ $($in particular, if \ ${\rm char} \ k= 0)$,
then it admits no bialgebra structures, hence no Hopf algebra structures.
\end{corollary}

\section{\bf Bi-Frobenius algebra structures on quantum exterior algebras in two variables}

In the special case of quantum exterior algebras in two variables, Theorem \ref{mainthm} has a stronger form.
Namely, let  $A_q = A(q, 2, 2) = k\langle x_1, x_2\rangle/\langle x_1^2, \ x_2^2, \ x_1x_2 + qx_2x_1\rangle$.
The main result of this section is as follows.

\begin{theorem} \label{bifroontwovar} \ If $A_q$ admits a bi-Frobenius algebra structure,  then $q=\pm 1.$

Conversely, if \ $\sqrt{-1}\in k$ and \ $q=\pm 1,$  then $A_q$ admits a bi-Frobenius algebra structure.

Moreover, $A_q$ admits a Hopf algebra structure if and only if ${\rm char} \ k = 2$.
\end{theorem}

\subsection {\bf The condition on $q$} Denote a basis of $A_q$ by
$$b_0 =1, \ \ b_1 = x_1, \ \ b_2 = x_2, \ \ b_3= x_1x_2.$$
Consider the dual basis of $A_q^*$:  \ $b_i^*, \ i = 0, 1, 2, 3$, where \ $b_i^*(b_j) = \delta_{ij}, \ \forall \ 0\le i, j\le 3$. By Lemma \ref{frobeniushomomorphism} one has

\begin{fact}\label{2var} All the Frobenius homomorphisms of
$A_q$ are \ $\phi= c_0 b_0^* + c_1 b_1^* + c_2 b_2^* + c_3 b_3^*$, where $c_i\in k \ \text{and} \ c_3\neq 0$.
\end{fact}

The following lemma is not a consequence of Theorem \ref{mainthm}, since the coalgebra structure in  Theorem \ref{mainthm} is special, as given in (3.1); but in the following lemma there are no restrictions on coalgebra structures.

\begin{lemma} \label{bifroontwovar1} If \ $A_q$ admits a bi-Frobenius algebra structure, then \ $q=\pm 1.$
\end{lemma}
\begin{proof} \ Suppose that \ $(A_q, \ \phi, \ t, \ S)$ is a bi-Frobenius algebra, with comultiplication $\Delta$ and counit $\varepsilon$.
By Fact \ref{2var}, \
$\phi = \sum\limits_{0\le i\le 3}c_ib_i^*$ with $c_3\ne 0$. By Lemma \ref{dt}(3), \ $t$ is a right integral, i.e., \ $t a = t \ \varepsilon(a)$ for any \ $a\in A_q$.
Then \  $tx_1 = tx_2 = 0$,  and hence \ $t = cx_1x_2$ with $0\ne c\in k$.

By definition, the identity \ $1$ is a group-like element.  Write
\begin{equation*}
  \Delta(b_i)=1\otimes b_i + b_i\otimes 1 +\Sigma_i,  \ \ 1\le  i \le  3,
\end{equation*}
where  $$\Sigma_1 = \sum_{1\le  i,j\le  3}a_{ij}b_i\otimes b_j, \ \ \ \Sigma_2 = \sum_{1\le  i,j\le  3}b_{ij}b_i\otimes b_j, \ \ \ \Sigma_3 = \sum_{1\le  i,j\le  3}c_{ij}b_i\otimes b_j
$$
with all \ $a_{ij}, \ b_{ij}, \ c_{ij} \in k$. By Lemma \ref{dt}(3), \ $\phi$ is a right cointegral, i.e., \ $\phi(a)1 = \sum \phi(a_1)a_2$  for any \ $a\in A_q$. Thus
$\phi(t)1 =\sum\phi(t_1)t_2 = S(1)=1$. On the other hand,
$$\phi(t) = (\sum\limits_{0\le i\le 3}c_i b_i^*)(cx_1x_2) = cc_3,$$
so \  $c=1/c_3$. By Lemma \ref{dt}(4), without lose of generality, one can assume $c_3 = 1 = c$, i.e.,
$A_q$ has a bi-Frobenius algebra structure $(A_q, \ \phi = c_0b_0^*+c_1b_1^*+c_2b_2^*+b_3^*, \ t = x_1x_2, \ S)$.

By $a \leftharpoonup  f = \sum f(a_1)a_2, \ \forall \ a\in A_q, \ f\in A_q^*$, and
$$\Delta(t) = \Delta(x_1x_2)=1\otimes x_1x_2 + x_1x_2\otimes 1 + \sum_{1\le  i, j\le  3} c_{ij}b_i\otimes b_j = \sum t_1\otimes t_2$$   one has
\begin{align*}
    &x_1x_2 \leftharpoonup b_0^*= x_1x_2 \\
    &x_1x_2 \leftharpoonup b_1^*=c_{11}x_1+c_{12}x_2+c_{13}x_1x_2 \\
    &x_1x_2 \leftharpoonup b_2^*=c_{21}x_1+c_{22}x_2+c_{23}x_1x_2 \\
    &x_1x_2 \leftharpoonup b_3^*=c_{31}x_1+c_{32}x_2+c_{33}x_1x_2+1.
\end{align*}
Since \ $(A_q, \ t)$ is a Frobenius coalgebra, \ $A_q = t \leftharpoonup A_q^* = x_1x_2 \leftharpoonup A_q^*$. It follows that $\left(
    \begin{smallmatrix}
     c_{11} & c_{12} \\
     c_{21} & c_{22}
    \end{smallmatrix}
  \right)$ is an invertible matrix, i.e.,  $c_{11}c_{22}\neq c_{12}c_{21}$.

By definition $S: A\longrightarrow A$ is the $k$-linear map  given by $S(a)= \sum\phi(t_{1}a)t_{2}.$ Thus
%$$S(1)  = \sum(c_0f_0+c_1f_1+c_2f_2+f_3)(t_{1})t_{2} = c_0 x_1x_2 + 1 +\sum\limits_{1\le i, j\le 3} c_{ij}c_ib_j;$$
\begin{align*} S(x_1) &= \sum(c_0b_0^*+c_1b_1^*+c_2b_2^*+b_3^*)(t_{1}x_1)t_{2} \\ & = c_1x_1x_2 + (c_0b_0^*+c_1b_1^*+c_2b_2^*+b_3^*)(x_2x_1) \sum_{1\le  j\le  3}c_{2j}b_j
\\ & = c_1x_1x_2 + b_3^*(x_2x_1)\sum_{1\le  j\le  3}c_{2j}b_j \\ & = -\frac{1}{q}[c_{21}x_1+c_{22}x_2+(c_{23}-qc_1)x_1x_2];\end{align*}
\begin{align*}S(x_2) & = \sum(c_0b_0^*+c_1b_1^*+c_2b_2^*+b_3^*)(t_{1}x_2)t_{2} \\ & = c_2x_1x_2 + (c_0b_0^*+c_1b_1^*+c_2b_2^*+b_3^*)(x_1x_2) \sum_{1\le  j\le  3}c_{1j}b_j
\\ & = c_2x_1x_2 + \sum_{1\le  j\le  3}c_{1j}b_j \\ & = c_{11}x_1+c_{12}x_2+(c_{13}+c_2)x_1x_2;\end{align*}
$$S(x_1x_2)  = \sum(c_0b_0^*+c_1b_1^*+c_2b_2^*+b_3^*)(t_{1}x_1x_2)t_{2} = b_3^*(x_1x_2)x_1x_2 =x_1x_2.$$

 \noindent So one has
$$\left\{
\begin{array}{l}
    %S(1)=c_0 x_1x_2 + 1 +\sum\limits_{1\le i, j\le 3} c_{ij}c_ib_j; \\
    S(x_1)=-\frac{1}{q}[c_{21}x_1+c_{22}x_2+(c_{23}-qc_1)x_1x_2]; \\
    S(x_2)=c_{11}x_1+c_{12}x_2+(c_2+c_{13})x_1x_2; \\
    S(x_1x_2)=x_1x_2.
\end{array}
\right. \eqno{(5.1)}$$

Since $S$ is a coalgebra anti-homomorphism, i.e.,  $\Delta(S(a))=\sum S(a_{2})\otimes S(a_{1})$ for any $a\in A_q$, one  has
\begin{equation*}
  \Delta(S(x_1x_2))=\Delta(x_1x_2)=1\otimes x_1x_2+x_1x_2\otimes 1+\Sigma_3,\\
\end{equation*}
and
\begin{align*}
  \sum S(t_{2})\otimes S(t_{1})
  &= S(1)\otimes S(x_1x_2)+S(x_1x_2)\otimes S(1)+ \sum_{1\le i,j\le 3} c_{ij}S(b_j)\otimes S(b_i)\\
  &= 1\otimes x_1x_2+x_1x_2\otimes 1+\sum_{1\le i,j\le 3} c_{ij}S(b_j)\otimes S(b_i).
\end{align*}
Thus  \  $\Sigma_3=\sum\limits_{1\le i,j\le 3} c_{ij}S(b_j)\otimes S(b_i)$, that is
\begin{equation*}
  \sum\limits_{1\le i, j\le 3} c_{ij}b_i\otimes b_j = \sum\limits_{1\le i, j\le 3} c_{ij}S(b_j)\otimes S(b_i).
\end{equation*}

\noindent Using (5.1) and comparing the coefficients of $x_1\otimes x_1, \ \ x_1\otimes x_2$  and \ \ $x_2\otimes x_1$, one gets
$$c_{11}=
  \frac{1}{q^2}c_{11}c_{21}^2 -\frac{1}{q}c_{11}c_{12}c_{21} -\frac{1}{q}c_{11}c_{21}^2 + c_{22}c_{11}^2,\eqno(5.2)$$
$$c_{12} =
  \frac{1}{q^2}c_{11}c_{21}c_{22}-\frac{1}{q}c_{11}c_{12}c_{22}-\frac{1}{q}c_{12}c_{21}^2+c_{22}c_{11}c_{12},\eqno(5.3)$$
$$c_{21} =
  \frac{1}{q^2}c_{11}c_{21}c_{22}-\frac{1}{q}c_{21}c_{12}^2-\frac{1}{q}c_{11}c_{21}c_{22}+c_{11}c_{12}c_{22}.\eqno(5.4)$$

By \ $S(x_1)^2=0$ one has \ $(c_{21}c_{22})(1-q)=0$. By \ $S(x_2)^2=0$ one has  \ $c_{11}c_{12}(1-q)=0$. By \ $S(x_1 x_2)=S(x_2) S(x_1)$ one has \ $c_{12}c_{21}-qc_{11}c_{22}=q^2.$
All together one has the following:

{\rm (i)} \ \ \ $c_{11}c_{22}\neq c_{12}c_{21};$

{\rm (ii)} \ \  $(c_{21}c_{22})(1-q)=0;$

{\rm (iii)} \ \ $c_{11}c_{12}(1-q)=0;$

{\rm (iv)} \ \ $c_{12}c_{21}-qc_{11}c_{22}=q^2.$

Now,  assume that $q\neq 1$. By (ii) and (iii), \ $c_{21}c_{22}=0 = c_{11}c_{12}.$ By (5.2) and (5.3) one has
$$c_{11} = \frac{1}{q^2}c_{11}c_{21}^2  -\frac{1}{q}c_{11}c_{21}^2 + c_{22}c_{11}^2\eqno(5.5)$$
$$c_{12} = -\frac{1}{q}c_{12}c_{21}^2\eqno(5.6)$$

By $c_{11}c_{22}\neq c_{12}c_{21}$ and $c_{21}c_{22}=0 = c_{11}c_{12}$, there are two situations:

\begin{enumerate}
  \item \ $c_{11}= 0 =c_{22}$ \ and \ $c_{12}\neq 0 \ne c_{21};$
  \item \ $c_{12}=0= c_{21}$ \ and \ $c_{11}\neq 0\ne c_{22}.$
\end{enumerate}

If (1) happens, then by  (iv) one has \  $c_{12}c_{21}=q^2$. By $(5.5)$ and (5.6) one gets  $c_{21}^2=-q = c_{12}^2$. Thus $q^4 = q^2$, and hence $q=-1$.

If (2) happens, then by  (iv) one has \ $c_{11}c_{22}=-q$.  By $(5.5)$ one gets $c_{11}c_{22}=1$. Thus $q=-1$.

 In conclusion, if $A_q$ has a bi-Frobenius algebra structure, then either $q =1$, or $q = -1$. That is, if  $q^2 \ne 1$, then
$A_q$ can not be a bi-Frobenius algebra.\end{proof}

\subsection{\bf Proof of Theorem \ref{bifroontwovar}.} \ In Lemma \ref{bifroontwovar1} we already know that if $A_q$ admits a bi-Frobenius algebra structure, then $q=\pm 1.$
Conversely, suppose that \ $k$ contains \ $\sqrt{-1}$ and  that $q=\pm 1.$  Then by Theorem \ref{mainthm}, $A_q$ has a bi-Frobenius algebra structure, with \ $\Delta$ and \ $\varepsilon$ given by (3.1).

Moreover, by Corollary \ref{hopf},  \ $A_q$ admits a Hopf algebra structure if and only if \ ${\rm char} \ k = 2$.  \hfill $\square$

\subsection{\bf A condition on field $k$} \ We point out that, if \ $k$ \ does not contain  $\sqrt{-1}$, then the exterior algebra
with two variables \ $A_1 = k\langle x_1, x_2\rangle/\langle x_1^2, \ x_2^2, \ x_1x_2 + x_2x_1\rangle$ admits no bi-Frobenius algebra structures.

\begin{proposition} \label{twovar} \ If \ $\sqrt{-1}\notin k$, then  the exterior algebra
with two variables \ $A_{1}$ admits no  bi-Frobenius algebra structures.
\end{proposition}
\begin{proof} \ Let $k$ be a field such that $\sqrt{-1}\notin k$.  Then \  ${\rm char}\ k \ne 2$. Assume  otherwise that $A_{1}$ has a bi-Frobenius algebra structure \ $(A_1, \ \phi, \ t, \ S)$, with comultiplication $\Delta$ and counit $\varepsilon$. As in the proof of Lemma \ref{bifroontwovar1} one gets (5.2),  (5.3), and (5.4), using $q= 1$ and rewriting them one has
$$c_{11} = -c_{11}c_{12}c_{21}+ c_{22}c_{11}^2= c_{11}(c_{11}c_{22}-c_{12}c_{21}) \eqno(5.2')$$
$$c_{12} = c_{11}c_{21}c_{22}-c_{12}c_{21}^2 = c_{21}(c_{11}c_{22}-c_{12}c_{21})\eqno(5.3')$$
$$c_{21} =-c_{21}c_{12}^2+c_{11}c_{12}c_{22}= c_{12}(c_{11}c_{22}-c_{12}c_{21}), \eqno(5.4')$$
and the condition (iv) becomes
$$c_{12}c_{21}-c_{11}c_{22}=1.\eqno({\rm iv}')$$
By substituting \ $({\rm iv}')$ into  \ $(5.2')$, $(5.3')$ and $(5.4')$, one gets $c_{11}=c_{22}=0$, and then \ $c_{12}^2=-1$. But $c_{12}\in k$, this contradicts
the assumption $c_{12} = \sqrt{-1} \notin k$.  This completes the proof.  \end{proof}

\section{\bf Non-isomorphic bi-Frobenius algebra structures on a commutative algebra}

We give two coalgebra structures on a complete intersection ring, such that it admits non-isomorphic bi-Frobenius algebra structures.

\subsection{\bf Coalgebra structure I}  Taking ${\bf q = -1}$ \ (i.e., \ $q_{ij} = -1$ for \ $1\le i, j\le n$) \ in quantum complete intersections, one gets  complete intersection ring
\[A = A({\bf -1}, a_1, \cdots, a_n) = k\langle x_1, \cdots, x_n\rangle/\langle x_i^{a_i}, \ x_i x_j-x_j x_i, \ \ 1\le i, \ j\le n\rangle.
\]
Thus \ $\mathcal B = \{x_{\mathbf{v}}= x_1^{v_1}\cdots x_n^{v_n} \ |  \ \mathbf{v}\in V\}$  is a basis of $A$,
where \ $V=\{ \mathbf{v} = (v_1, \cdots, v_n)\in \mathbb{N}_0^n \ | \ v_i\le a_i-1,  \ 1\le i\le n\}.$ Inspired by
the construction of path coalgebra of quivers (see D. Simson \cite {Simson2001}; also C. Cibils and M. Rosso \cite{CR2002}), we consider the $k$-linear maps \ $\varepsilon: A\longrightarrow k$ and \ $\Delta: A \longrightarrow A\otimes A$  given by

$$\left\{\begin{array}{l}  \varepsilon(x_{\mathbf{v}})=\delta_{\mathbf{v,0}}, \ \ \forall \ \mathbf{v}\in V;
\\ \\ \Delta(x_\mathbf{v})=\sum\limits_{\substack{\mathbf{v_1+v_2=v}, \\ \mathbf{v_1,v_2}\in V}} x_{\mathbf{v_1}}\otimes x_{\mathbf{v_2}}, \ \ \forall \ \mathbf{v}\in V.
\end{array}
\right. \eqno{(6.1)}$$

\begin{theorem} \label{pathcoalg} \ With counit \ $\varepsilon$ and comultiplication \ $\Delta$ as in $(6.1)$, the complete intersection ring \ $A = A({\bf -1}, a_1, \cdots, a_n)$ forms a bi-Frobenius algebra
\ $(A,  \ (x_{\mathbf{a-1}})^*, \ x_{\mathbf{a-1}}, \ {\rm Id})$.

Moreover, in general this bi-Frobenius algebra is not isomorphic   $($as a bi-Frobenius algebra$)$ to the one in {\rm Corollary \ref{cormainthm}}, with coalgebra structure given by $(3.1)$.
\end{theorem}
\begin{proof} \ It is straightforward to verify that  $(A, \ \Delta, \ \varepsilon)$ is a coalgebra,  the identity \ $1$ is a group-like element, and that $\varepsilon$ is an algebra map.
Put \ $t = x_{\bf a-1 }$. Since \ $A$ is commutative and cocommutative,  \ ${\rm Id}$ is an algebra anti-homomorphism and a coalgebra anti-homomorphism.  To prove  \ $(A,  \ (x_{\mathbf{a-1}})^*, \ x_{\mathbf{a-1}}, \ {\rm Id})$ is a bi-Frobenius algebra,
by Lemma \ref{lemDoi2002},
it remains to check $$\sum (x_{\bf a-1})^*(t_1x_{\bf v})t_2 = x_{\bf v}, \ \ \forall \ {\bf v}\in V.$$
In fact, by  \ $\Delta (t)=\Delta(x_{\bf a-1}) = \sum\limits_{\substack{\mathbf{v_1+v_2=a-1},\\ \mathbf{v_1,v_2}\in V}} x_{\mathbf{v_1}}\otimes x_{\mathbf{v_2}}$,
one has
\begin{align*}
\sum (x_{\bf a-1})^*(t_1x_{\bf v})t_2 = (x_{\bf a-1})^*(x_{\bf a-1-v} \ x_{\bf v}) \ x_{\bf v} = x_{\bf v}, \ \forall \ {\bf v}\in V.
\end{align*}
This proves \ $(A,  \ (x_{\mathbf{a-1}})^*, \ x_{\mathbf{a-1}}, \ {\rm Id})$ is a bi-Frobenius algebra.
For convenience, denote this bi-Frobenius algebra by $B_1$.

When \  $\sqrt{-1}\in k$, to show that \ $B_1$ is not isomorphic to the bi-Frobenius algebra $B_2 = (A, \ (x_{{\bf a-1}})^*, \ x_{{\bf a-1}}, \ S)$ in {\rm Theorem \ref{mainthm}}, we consider the spaces \ $P(B_1)$ and $P(B_2)$,  of the primitive elements of $B_1$ and $B_2$, respectively.

By Lemma \ref{frobeniuscoalg},  \ dim $P(B_2) = (\prod\limits_{1\le i\le n}a_i) -2.$

On the other hand, each \ $x_i$ is a primitive element of $B_1$.
We claim that ${\rm dim} \ P(B_1) =n$. Otherwise, there exists a primitive element
\ $x=\sum\limits_{{\bf v}\in W}c_{\bf v} \ x_{\bf v}$, where $W$ is a non empty subset of \ $V$ such that \ $c_{\bf v}\ne 0$ and \ ${\bf v\ne e_i}$ for \ $ 1\le i\le n$, for all \ ${\bf v}\in W$. From the proof below, one can see that, without loss of generality, one may assume that ${\bf 0}\notin W$. From the right hand side of the equality
$$\sum\limits_{{\bf v}\in W}c_{\bf v} \ \Delta(x_{\bf v}) = \Delta(x)  = 1\otimes x + x\otimes 1 = \sum\limits_{{\bf v}\in W}c_{\bf v} \ (1\otimes x_{\bf v}) + \sum\limits_{{\bf v}\in W}c_{\bf v} \ (x_{\bf v}\otimes 1)$$
and since ${\bf 0}\notin W$, one sees that \ $\Delta(x)$ is a \ $k$-linear combination of $2|W|$ elements in the basis $\mathcal B\otimes \mathcal B$ of $A\otimes A$, with each coefficient nonzero, where $|W|$ is the number of elements in $W$. While by (6.1), each \ $\Delta(x_{\bf v})$ in the left hand side is a \ $k$-linear combination of $t$ elements in the basis $\mathcal B\otimes \mathcal B$, where $t > 2$.
For ${\bf u\ne v}\in V$, there are no the same summands in $\Delta(x_{\bf u})$ and $\Delta(x_{\bf v})$, since $x_{\bf u_1}\otimes x_{\bf u_2} = x_{\bf v_1}\otimes x_{\bf v_2}$ implies  ${\bf u = u_1+u_2=v_1+v_2=v}$. In this way we see that $\Delta(x)$ is a \ $k$-linear combination of $s$ elements in the basis $\mathcal B\otimes \mathcal B$, with each coefficient nonzero, where $s > 2|W|$. A contradiction! This proves dim $P(B_1)= n$.

Thus, one has \ ${\rm dim} \ P(B_2) = (\prod\limits_{1\le i\le n}a_i) -2 > n = {\rm dim} \ P(B_1)$, except the case \ $n = 2 = a_1 = a_2$. So \ $B_1\ncong B_2$ as coalgebras,  except \ $n = 2 = a_1 = a_2$.  This completes the proof. \end{proof}

\begin{remark} \label{twovar2} \ If ${\bf q} \ne {\bf -1}$ in {\rm Theorem \ref{pathcoalg}} $($i.e., the algebra \ $A$ considered is not complete intersection ring$)$,
then the corresponding result is no longer true. For example, when \ $\sqrt{-1}\notin k$,
the exterior algebra with two variables \ $A_1 = k\langle x_1, x_2\rangle/\langle x_1^2, \ x_2^2, \ x_1x_2 + x_2x_1\rangle$ admits no bi-Frobenius algebra structures. See {\rm Proposition \ref{twovar}}.
\end{remark}

\subsection{\bf Coalgebra structure II} In Subsection 6.1, taking \ $(a_1, \ \cdots, \ a_n) = (2, \ \cdots, \ 2)$,  namely
$$A = A({\bf -1}, \ 2, \ \cdots, \ 2) = k\langle x_1, \cdots, x_n\rangle/\langle x_i^{2}, \ x_i x_j- x_j x_i,  \ 1\le i, \ j \le n\rangle.$$
Consider the $k$-linear maps \ $\varepsilon: A\longrightarrow k$ and \ $\Delta: A \longrightarrow A\otimes A$  given by

$$\left\{\begin{array}{l}
\varepsilon(x_\mathbf{v})=\delta_{\mathbf{0,v}}, \ \ \forall \ \mathbf{v}\in V; \\ \\
\Delta(x_\mathbf{v})=\sum\limits_{\substack{\mathbf{v_1, v_2\le  v}\\ \mathbf{v\le  v_1+v_2}}}(-1)^{|\mathbf{v_1+v_2-v}|} \ x_{\mathbf{v_1}}\otimes x_{\mathbf{v_2}}, \ \ \forall \ \mathbf{v}\in V.
\end{array}
\right. \eqno{(6.2)}$$

\begin{theorem} \label{2truncatedpolyalg} \ With counit \ $\varepsilon$ and comultiplication \ $\Delta$ given by $(6.2)$,
\ $(A, \ \phi, \ t = x_1\cdots x_n, \ {\rm Id})$ is a bi-Frobenius algebra, where \ $\phi = \sum\limits_{\mathbf{v}\in V} (x_\mathbf{v})^*$.

Moreover, this bi-Frobenius algebra has no primitive elements, and hence it can not be isomorphic to the bi-Frobenius algebras given in {\rm Corollary \ref{cormainthm}} and in {\rm Theorem \ref{pathcoalg}}.
\end{theorem}
\begin{proof} \ One directly verify that $(A, \ \Delta, \ \varepsilon)$ is a coalgebra:
\begin{align*}
  (\Delta\otimes {\rm Id})\Delta(x_\mathbf{v})
  &=
  \sum_{\substack{\mathbf{v_1, v_2\le  v}\\ \mathbf{v\le  v_1+v_2}}}(-1)^{|\mathbf{v_1+v_2-v}|} \ \Delta(x_{\mathbf{v_1}})\otimes x_{\mathbf{v_2}}\\
  &=
  \sum_{\substack{\mathbf{v_1, v_2, v_3\le  v}\\ \mathbf{v\le  v_1+v_2+v_3}}}(-1)^{|\mathbf{v_1+v_2+v_3-v}|} \ x_{\mathbf{v_1}}\otimes x_{\mathbf{v_2}}\otimes x_{\mathbf{v_3}}\\
  &=
  ({\rm Id}\otimes\Delta)\Delta(x_\mathbf{v}), \ \ \forall \ \mathbf{v}\in V
\end{align*}
and
\begin{align*}
  (\varepsilon\otimes {\rm Id})\Delta(x_\mathbf{v})
  &=
  \sum_{\mathbf{v\le  v_2\le  v} }(-1)^{|\mathbf{v_2-v}|} \ 1\otimes x_{\mathbf{v_2}}= 1\otimes x_\mathbf{v} = x_\mathbf{v} \\
& = x_\mathbf{v}\otimes 1 =
  \sum_{\mathbf{v\le  v_1\le  v} }(-1)^{|\mathbf{v_1-v}|} \ x_{\mathbf{v_1}}\otimes 1 = ({\rm Id}\otimes \varepsilon)\Delta(x_\mathbf{v}), \ \ \forall \ \mathbf{v}\in V.
\end{align*}

Put \ $t = x_1\cdots x_n$. Since \ $A$ is commutative and cocommutative,  \ ${\rm Id}$ is an algebra anti-homomorphism and a coalgebra  anti-homomorphism.
By Lemma \ref{lemDoi2002}, it remains to prove $$\sum \phi(t_1x_{\bf u})t_2 = x_{\bf u}, \ \ \forall \ {\bf u}\in V.$$
In fact, by
\[
  \Delta(t)
  =\Delta(x_1\cdots x_n)=\sum_{\substack{\mathbf{v_1, v_2}\le (1, \cdots, 1)\\ (1, \cdots, 1) \le \mathbf{v_1 + v_2}}}(-1)^{|\mathbf{v_1+v_2-1}|} \ x_{\mathbf{v_1}}\otimes x_{\mathbf{v_2}}
\]
one has
\begin{align*}
\sum \phi(t_1x_{\bf u})t_2 & =
    \sum_{\substack{\mathbf{v_1, v_2}\le  (1, \cdots, 1)\\ (1, \cdots, 1)\le  {\bf v_1+v_2}}}(-1)^{\bf |{v_1}+{v_2}-1|} \ \phi(x_{\mathbf{v_1}} \ x_\mathbf{u}) \ x_{\mathbf{v_2}}, \ \forall \ \mathbf{u}\in V.
  \end{align*}
If \ $v_{1i} > 1-u_i $ for some $i$ (here \ $v_{1i}$ means the \ $i$th component of \ ${\bf v_1}$), then \ $v_{1i} + u_i \ge 2$, and then \ $x_\mathbf{v_1} \ x_\mathbf{u} = x_{\mathbf{v_1+u}} =0$. Thus,  we only need to consider those \ $\mathbf{v_1}$ with
$\mathbf{v_1}\le \mathbf{1-u}$, and hence \ ${\bf v_2\ge 1-v_1\ge u}$. Fixing ${\bf v_2}$, then the coefficient of \ $x_{\mathbf{v_2}}$ in the sum above is (note that \ $\phi = \sum\limits_{\mathbf{v}\in V} (x_\mathbf{v})^*$)
\begin{align*}
  \sum_{\mathbf{1-v_2 \le  v_1\le  1-u}}(-1)^{|{\bf v_1+v_2 -1}|} \ \phi(x_\mathbf{v_1 + u})
  &=  \sum_{\mathbf{1-v_2 \le  v_1\le  1-u}}(-1)^{\bf |{v_1}+{v_2}-1|}\\
  &=\sum_{\mathbf{0\le  v_1+v_2-1\le  v_2-u}}(-1)^{\bf |{v_1}+{v_2}-1|}.
\end{align*}

 If $\mathbf{v_2 = u}$, then the vector \ ${\bf v_1}$ in the sum is unique, so the coefficient of \ $x_{\mathbf{v_2}}$ in the sum  is
$$\sum\limits_{\mathbf{0\le  v_1+v_2-1\le  v_2-u}}(-1)^{|{\bf v_1+v_2}-1|}= (-1)^{|{\bf 0}|} =1.$$

Assume that \ ${\bf v_2}\ne \mathbf{u}$. Set \ $m = |{\bf v_2-u}|$. Note that \  ${\bf v_2-u}$ is
a vector with \ $m$ components being $1$ and \ $n-m$ components being $0$. We count the number of vectors
\ ${\bf v_1}$ in the sum via $|{\bf v_1+v_2-1}|$ as follows.

When \ $|{\bf v_1+v_2-1}|=0$, then the choice of ${\bf v_1}$ is unique,
i.e., ${\bf v_1 = 1-v_2}$. That is, the number \ of ${\bf v_1}$ is $\binom{m}{0}=1$.

When $|{\bf v_1+v_2-1}|=1$, \ since \ $\mathbf{v_1+v_2-1\le  v_2-u}$ and ${\bf v_2-u}$ has
\ $m$ components being $1$, it follows that the number of the choices of ${\bf v_1}$ such that $|{\bf v_1+v_2-1}|=1$ is just \ $\binom{m}{1}$;

In general, when \ $|{\bf v_1+v_2-1}|=i$ with \ $1\le i\le n$, then the number of the choices of ${\bf v_1}$ such that $|{\bf v_1+v_2-1}|=i$ is just \ $\binom{m}{i}$.

All together, if \ ${\bf v_2}\ne \mathbf{u}$, then we see that the coefficient of \ $x_{\mathbf{v_2}}$ in the sum above is
\[
\sum_{\mathbf{0\le  v_1+v_2-1\le  v_2-u}}(-1)^{\bf |{v_1}+{v_2}-1|}=\sum_{i=0}^{m} \binom{m}{i}(-1)^i=0.
\]

In this way we see that
\begin{align*}\sum \phi(t_1x_{\bf u})t_2 =
    \sum_{\substack{\mathbf{v_1, v_2}\le  (1, \cdots, 1)\\ (1, \cdots, 1)\le  {\bf v_1+v_2}}}(-1)^{\bf |{v_1}+{v_2}-1|} \ \phi(x_{\mathbf{v_1}} \ x_\mathbf{u}) \ x_{\mathbf{v_2}}  = x_{\bf u}, \ \ \forall \ {\bf u}\in V.\end{align*}
This proves that \ $(A, \ \phi, \ x_\mathbf{1} = x_1\cdots x_n, \ {\rm Id})$ is a bi-Frobenius algebra.

Using the similar idea as in the proof of Theorem \ref{pathcoalg}, one claim that
this bi-Frobenius algebra has no primitive elements. For completeness, we include a proof.

In fact, assume otherwise that \ $x=\sum\limits_{{\bf v}\in W}c_{\bf v} \ x_{\bf v}$ is a primitive element, where $W$ is a non empty subset of \ $V$ such that \ $c_{\bf v}\ne 0$ for all \ ${\bf v}\in W$.

First, assume that \ ${\bf 0}\notin W$. From the right hand side of the equality
$$\Delta(x) = 1\otimes x + x\otimes 1 = \sum\limits_{{\bf v}\in W}c_{\bf v} \ (1\otimes x_{\bf v}) + \sum\limits_{{\bf v}\in W}c_{\bf v} \ (x_{\bf v}\otimes 1)$$ and since ${\bf 0}\notin W$, one sees that \ $\Delta(x)$ is a \ $k$-linear combination of $2|W|$ elements in the basis $\mathcal B\otimes \mathcal B$ of $A\otimes A$, with each coefficient nonzero. On the other hand,
$$\Delta(x) = \sum\limits_{{\bf v}\in W}c_{\bf v} \ \Delta(x_{\bf v})$$ and by (6.2) one has
$$\Delta(x_\mathbf{v})=\sum\limits_{\substack{\mathbf{v_1, v_2\le  v}\\ \mathbf{v\le  v_1+v_2}}}(-1)^{|\mathbf{v_1+v_2-v}|} \ x_{\mathbf{v_1}}\otimes x_{\mathbf{v_2}}.$$
Since ${\bf 0}\notin W$ and \ $\Delta(x_{\bf v})$ at least includes the terms:
$$x_{\bf v}\otimes x_{\bf v}, \ \ x_{\bf v -e_i}\otimes x_{\bf v}, \ \ x_{\bf v}\otimes x_{\bf v-e_i}$$ for some \ ${\bf e_i}$, it follows that \ $\Delta(x_{\bf v})$ is a \ $k$-linear combination of $t$ elements in the basis $\mathcal B\otimes \mathcal B$, where $t \ge 3$.

For ${\bf u\ne v}\in V$, there are no the same terms in $\Delta(x_{\bf u})$ and $\Delta(x_{\bf v})$. Actually, if $x_{\bf u_1}\otimes x_{\bf u_2} = x_{\bf v_1}\otimes x_{\bf v_2}$, then
$${\bf u_1} = {\bf v_1}, \ \ {\bf u_2} = {\bf v_2}, \ \ {\bf u_1,u_2\le u\le u_1+u_2}, \ \ {\bf v_1, v_2\le v\le v_1+v_2}.$$ Since $a_i=2$ for $1\le i\le n$, one has
\[
u_i=\begin{cases}
    0 &\text{if}\ \  {u_1}_i={u_2}_i=0;\\
    1 &\text{if \ \ otherwise}
\end{cases}
\]
and
\[
v_i=\begin{cases}
    0 &\text{if}\ \  {v_1}_i={v_2}_i=0;\\
    1 &\text{if \ \ otherwise}. \end{cases}
\]
It follows that \ ${\bf u=v}.$  In this way we see that $\Delta(x)$ is a \ $k$-linear combination of $s$ elements in the basis $\mathcal B\otimes \mathcal B$, with each coefficient nonzero, where $s \ge 3|W|$. A contradiction!

Next, assume that \ ${\bf 0}\in W$. Since \ $1$ is not a primitive element, \ $|W|\ge 2$. By the similar argument as above, one sees that
$$\sum\limits_{{\bf v}\in W}c_{\bf v} \ \Delta(x_{\bf v}) = \Delta(x) = \sum\limits_{{\bf v}\in W}c_{\bf v} \ (1\otimes x_{\bf v}) + \sum\limits_{{\bf v}\in W}c_{\bf v} \ (x_{\bf v}\otimes 1).$$
The left hand side is a \ $k$-linear combination of $s$ elements in the basis $\mathcal B\otimes \mathcal B$, with each coefficient nonzero, where $s \ge 3(|W|-1) + 1 = 3|W| -2$; while
the right hand side is a \ $k$-linear combination of $2|W| -1$ elements in the basis $\mathcal B\otimes \mathcal B$, with each coefficient nonzero. A contradiction!

This completes the proof. \end{proof}

\begin{remark} \label{2truncatedpolyalg2} \ If $(a_1, \cdots, a_n) \ne (2, \cdots, 2)$ in {\rm Theorem \ref{2truncatedpolyalg}},
then the corresponding result is no longer true.

For example, let \ $A = k\langle x_1, x_2\rangle/\langle x_1^2, \ x_2^3, \ x_1x_2-x_2x_1\rangle$. Then \ $A$ becomes a coalgebra with comultiplication $\Delta$ and counit $\varepsilon$ as given by $(6.2)$,
but \ $(A, \ \phi= \sum\limits_{\mathbf{v}\in V} (x_\mathbf{v})^*, \ t = x_1  x_2^2, \ S)$ is not a bi-Frobenius algebra, where $S: A \longrightarrow  A$ is the $k$-linear map
given by $S(a)= \sum\phi(t_{1}a)t_{2}.$

In fact, in this case \ $V= \{(0,0)$, $(0,1)$, $(0,2)$, $(1,0)$, $(1,1)$, $(1,2)\}$,

\begin{align*}
    \Delta(t)&=\Delta(x_1x_2^2)= \sum_{{\substack {{\bf v_1,v_2}\le (1,2), \\ (1,2)\le {\bf v_1+v_2}}}}(-1)^{|{\bf v_1+v_2}-(1,2)|} \ x_{\bf v_1}\otimes x_{\bf v_2},
\end{align*}
and
\begin{align*}
    S(1)&=\sum_{{\substack {{\bf v_1,v_2}\le (1,2), \\ (1,2)\le {\bf v_1+v_2}}}}(-1)^{|{\bf v_1+v_2}-(1,2)|}\phi(x_{\bf v_1}) x_{\bf v_2}= 1+x_2^2.
\end{align*}
So \ $S$ is not an algebra homomorphism, and hence \ $A$ is not a bi-Frobenius algebra.
\end{remark}

{\bf Acknowledgements.} We thank Libin Li and Yanhua Wang for valuable suggestions concerning the
presentation of the results. We thank the referees for their helpful comments and suggestions.

\bibliographystyle{alpha}

\begin{thebibliography}{99}
\bibitem[1]{AH1998} Andruskiewitsch, N., Schneider, H.-J.: Lifting of quantum linear spaces and pointed Hopf algebras of order $p^3$, J. Algebra {\bf 209}(2), 658-691(1998)
\bibitem[2]{AH2002} Andruskiewitsch, N., Schneider, H.-J.:  Pointed Hopf algebras, New directions in Hopf algebras, Math. Sci. Res. Inst. Publ., {\bf 43}, Cambridge Univ. Press, Cambridge, 1-68(2002)
\bibitem[3]{AGP1997}Avramov, L. L., Gasharov, V. N., Peeva, I. V.: Complete intersection dimension, Publ. Math. I.H.E.S. {\bf 86}, 67-114(1997).
\bibitem[4]{Bergh2009}Bergh, P. A.: Ext-symmetry over quantum complete intersections, Arch. Math. {\bf 92}(6), 566-573(2009)
\bibitem[5]{BE2008} Bergh, P. A.: Erdmann, K.: Homology and cohomology of quantum complete intersections, Algebra $\&$ Number Theory {\bf 2}(5), 501-522(2008)
  \bibitem[6]{BE2011}Bergh, P. A., Erdmann, K.: The stable Auslander-Reiten quiver of a quantum complete intersection, Bull. Lond. Math. Soc. {\bf 43}(1), 79-90(2011)
  \bibitem[7]{BO2008} Bergh, P. A., Oppermann, S.: The representation dimension of quantum complete intersections, J. Algebra {\bf 320}(1), 354-368(2008)
  \bibitem[8]{BO20082} Bergh, P. A., Oppermann, S.: Cohomology of twisted tensor products, J. Algebra {\bf 320}(8)(2008), 3327-3338.
    \bibitem[9]{BGMS2005} Buchweitz, R.-O., Green, E., Madsen, D., Solberg, \O. : Finite Hochschild cohomology without finite global dimension, Math. Res. Lett. {\bf 12}(5-6), 805-816(2005)
  \bibitem[10]{CR2002} Cibils, C., Rosso, M.: Hopf quivers, J. Algebra, {\bf 254}, 241-251(2002)
  \bibitem[11]{Doi2002}Doi, Y.: Substructures of bi-frobenius algebras, J. Algebra {\bf 256}(2), 568-582(2002)
  \bibitem[12]{Doi2004}Doi, Y.: Bi-Frobenius algebras and group-like algebras, Hopf algebras, Lecture Notes in Pure and Appl. Math. {\bf 237}, 143-155(2004)
  \bibitem[13]{Doi2010}Doi, Y.: Group-like algebras and their representations, Comm. Algebra {\bf 38}(7), 2635-2655(2010)
  \bibitem[14]{DT2000}Doi, Y., Takeuchi, M.: Bi-Frobenius algebras, Contemp. Math. {\bf 267}, Amer. Math. Soc., Providence, RI, 67-97(2000)
  \bibitem[15]{Haim2007}Haim, M.: Group-like algebras and Hadamard matrices, J. Algebra {\bf 308}(1), 215-235(2007)
  \bibitem[16]{Hap1988} Happel, D.: Triangulated categories in representation theory of finite dimensional algebras, London Math.
Soc. Lecture Notes Ser. 119, Cambridge Uni. Press, 1988
\bibitem[17]{K1999}Kadison, L.: New examples of Frobenius extensions, Univ. Lecture Series 14, Amer. Math. Soc., Provindence, RI, 1999
  \bibitem[18]{LS1994} Liu, S., Schulz, R.: The existence of bounded infinite DTr-orbits, Proc. Amer. Math. Soc. 122, 1003-1005(1994)
  \bibitem[19]{L2011} Lorenz, M.: Some applications of Frobenius algebras to Hopf algebras, Groups, Algebras and Applications, Contemporay Math. {\bf 537}, Amer. Math. Soc., Providence, RI, 269-289(2011)
  \bibitem[20]{Manin1987} Manin, Yu. I.: Some remarks on Koszul algebras and quantum groups, Ann. Inst. Fourier(Grenoble) {\bf 37}, 191-205(1987)
  \bibitem[21]{Mar} Marczinzik, R.: On stable modules that are not Gorenstein projective, ArXiv:1709.01132v3(2017)
  \bibitem[22]{Moll2012}Moll, V. H.: Numbers and functions, From a classical-experimental mathematician's point of view, Student Math. Library 65. Amer. Math. Soc., Providence, RI, 2012
  \bibitem[23]{Oppermann2010}Oppermann, S.: Hochschild cohomology and homology of quantum complete intersections, Algebra $\&$ Number Theory {\bf 4}(7), 821-838(2010)
  \bibitem[24]{Radford1976}Radford, D. E.: The order of the antipode of a finite-dimensional Hopf algebra is finite, Amer. J. Math. {\bf 98}(2), 333-355(1976)
  \bibitem[25]{Ringel1996} Ringel, C. M.: The Liu-Schulz example, In: Representation theory of algebras, Canadian Math. Soc.  Conf. Proc. {\bf 18}, 587-600(1996)
\bibitem[26]{RZ2020} Ringel, C. M., Zhang, P.: Gorenstein-projective and semi-Gorenstein-projective modules, Algebra $\&$ Number Theory {\bf 14}(1), 1-36(2020)
\bibitem[27]{Simson2001}Simson, D.: Coalgebras, comodules, pseudocompact algebras and tame comodule type,  Colloq. Math. {\bf 90}(1), 101-150(2001)
\bibitem[28]{WC2007} Wang, Y. H., Chen, X. W.: Construct non-graded bi-Frobenius algebras via quivers, Sci. China Ser. A {\bf 50}(3), 450-456(2007)
\bibitem[29]{WL2014} Wang, Z., H., Li, L. B., Zhang, Y. H.: Green rings of pointed rank one Hopf algebras of nilpotent type, Algebra and Representation Theory {\bf 17}(6), 1901-1924(2014)
\bibitem[30]{WZ2004}Wang, Y. H., Zhang, P.: Construct bi-Frobenius algebras via quivers, Tsukuba J. Math. {\bf 28}(1), 215-221(2004)
\bibitem[31]{Y1996} Yamagata, K.: Frobenius algebras, Handbook of algebras, Vol. 1, Elsevier/North-Holland, Amsterdam, 841-887, 1996
\bibitem[32]{YZ2021}You, H. Y., Zhang, P.: Small modules over quantum complete intersections in two variables, J. Algebra and Its Applications. {\bf 20}(3), 2150047(2021)

\end{thebibliography}

\end{document}